\documentclass{amsart}
\usepackage{appendix}
\usepackage{tabularx}
\usepackage[numbers,sort&compress]{natbib}
\usepackage{color}
\usepackage{tikz}
\usetikzlibrary{shapes,arrows}

\tikzstyle{decision} = [diamond, draw, fill=blue!20, 
text width=6em, text badly centered, node distance=3cm, inner sep=0pt]
\tikzstyle{block} = [rectangle, draw, fill=blue!20, 
text width=9em, text centered, rounded corners, minimum height=4em]
\tikzstyle{block2} = [rectangle, draw, fill=yellow!20, 
text width=9em, text centered, rounded corners, minimum height=4em]
\tikzstyle{line} = [draw, -latex']
\tikzstyle{cloud} = [draw, ellipse,fill=red!20, node distance=3cm,
minimum height=4em]

\makeatletter

\newcommand{\Rmnum}[1]{\expandafter\@slowromancap\romannumeral #1@}
\makeatother

\newtheorem{theorem}{Theorem}[section]
\newtheorem{lemma}[theorem]{Lemma}
\newtheorem{corollary}[theorem]{Corollary}
\newtheorem{prop}[theorem]{Proposition}

\theoremstyle{definition}

\theoremstyle{remark}
\newtheorem{remark}[theorem]{Remark}

\numberwithin{equation}{section}
\allowdisplaybreaks[4]

\begin{document}


\tikzstyle{block} = [rectangle, draw, 
text width=5.5em, text centered, rounded corners, minimum height=2em]

\title{Real analyticity of the modified Laplacian coflow}


\author{Chuanhuan Li}
\address{ Shanghai Institute for Mathematics and Interdisciplinary Sciences (SIMIS), Shanghai 200433, China \newline
${\quad}$ Research Institute of Intelligent Complex Systems, Fudan University, Shanghai 200433, China}
\curraddr{}
\email{chli@simis.cn}
\thanks{}

\author{Yi Li}
\address{Center for Mathematics and Interdisciplinary Sciences, Fudan University, Shanghai 200433, China \newline
${\quad}$ Shanghai Institute for Mathematics and Interdisciplinary Sciences (SIMIS), Shanghai 200433, China}
\curraddr{}
\email{yilicms@simis.cn, yilicms@gmail.com}



\date{}

\dedicatory{}

\begin{abstract}
   Let $(M,\psi(t))_{t\in[0, T]}$ be a solution of the modified Laplacian coflow $\eqref{modified Laplacian coflow}$ with coclosed $G_{2}$-structures on a compact $7$-dimensional $M$. We improve Chen's Shi-type estimate \cite{Chen Shi-estimates} for this flow, and then show that $(M,\psi(t),g_{\psi}(t))$ is real analytic, where $g_{\psi}(t)$ is the  Riemannian metric associated to $\psi(t)$, which answers a question proposed by Grigorian in \cite{Grigorian-2020}. Consequently, we obtain the unique continuation results for this flow.

\end{abstract}

\maketitle

\section{Introduction}
Geometric flows play a crucial role in the study of geometric structures. The core idea is to start with a general geometric structure on a manifold and then utilize a flow to obtain a more specialized structure. Let $M$ be a $7$-dimensional manifold. Considering a closed $G_{2}$-structure on $M$, Bryant \cite{Bryant 2006} introduced the following Laplacian  flow:
\begin{equation}
  \left \{
       \begin{array}{rl}
          \partial_{ t}\varphi(t)&=\Delta_{\varphi(t)}\varphi(t),\\
           \varphi(0)&=\varphi,
       \end{array}
  \right.
  \label{The closed Laplacian flow}
\end{equation}
where $\Delta_{\varphi(t)}=dd^{\ast}_{\varphi(t)}+d^{\ast}_{\varphi(t)}d$ is the Hodge Laplacian of $g(t)$ and $\varphi$ is an initial closed $G_{2}$-structure. Here $g(t)$ is the associated Riemannian metric of $\varphi(t)$.

The flow $\eqref{The closed Laplacian flow}$ can be viewed as the gradient flow for the Hitchin functional introduced by Hitchin \cite{Hitchin 2000}
$$\mathcal{H}:[\overline{\varphi}]_{+}\longrightarrow\mathbb{R}^{+},\ \varphi\longmapsto\frac{1}{7}\int_{M}\varphi\wedge\psi=\int_{M}\ast_{\varphi}^{H}1.$$
Here $\overline{\varphi}$ is a closed $G_{2}$-structure on $M$, $[\overline{\varphi}]_{+}$ is the open subset of the cohomology class $[\overline{\varphi}]$ consisting of $G_{2}$-structures and $\psi$ is the {$4$-form associated to} $\varphi$. Any critical point of $\mathcal{H}$ gives rise to a torsion-free $G_{2}$-structure. In \cite{Bryant-Xu 2011}, {when the initial $3$-form $\varphi$
is closed and we evolve inside a fixed cohomology class,} Bryant and Xu obtained the short-time existence and uniqueness {to the solutions of {the} Laplacian flow}. Several important foundational analytic results were established by Lotay and Wei in \cite{Lotay-Wei Shi-estimate, L-W 2019, L-W 2019 2}, {including} derivative estimate, characterization of blow-up time, dynamical stability, and real analyticity. For the blow-up time with bounded Ricci curvature, the second author gave another proof in \cite{local curvature of G2}. For the blow-up time with bounded scalar curvature or bounded torsion tensor, see \cite{Fine-Yao 2018, LL G2 flow, Picard-Suan}.

In fact, a $G_{2}$-structure and the corresponding metric can also be defined using the 4-form $\psi=\ast_{\varphi}^{H}\varphi$ {and a given orientation}. Here $\ast^{H}$ is the Hodge star operator. Thus, Karigiannis, Mckay, and Tsui \cite{KMT2012} introduced the following  Laplacian coflow (with the opposite sign):
\begin{equation}
  \left \{
       \begin{array}{rl}
          \partial_{ t}\psi(t)&=\Delta_{\psi(t)}\psi(t),\\
           \psi(0)&=\psi,
       \end{array}
  \right.
  \label{The coflow}
\end{equation}
where $\Delta_{\psi(t)}$ is the Hodge Laplacian induced by $\psi(t)$ and $\psi$ is an initial coclosed $G_{2}$-structure. However, since the Laplacian coflow $\eqref{The coflow}$ is not weakly parabolic and the principal symbol of $\Delta_{\psi}\psi$ is indefinite, it cannot be modified to be parabolic in the same way as the Ricci flow. Therefore, Grigorian \cite{Grigorian-2013} introduced the modified Laplacian coflow to address this issue:
\begin{equation}
  \left \{
       \begin{array}{rl}
          \partial_{ t}\psi(t)&=\Delta_{\psi(t)}\psi(t)+2d\left[\left(A-{\rm tr} \big{(}\mathbf{T}(t)\big{)}\right)\varphi(t)\right],\\
           \psi(0)&=\psi,
       \end{array}
  \right.
  \label{modified Laplacian coflow}
\end{equation}
where ${\rm tr}\big{(}\mathbf{T}(t)\big{)}$ is the trace of the torsion $\mathbf{T}(t)$, $A$ is a nonnegative constant, and $\psi$ is an initial coclosed $G_{2}$-structure. {If restricted to coclosed $G_{2}$ structure($d\psi(t)=0$), then the modified Laplacian coflow $\eqref{modified Laplacian coflow}$ preserves the cohomology class of $\psi$. Under this condition,} Grigorian proved the short-time existence and uniqueness of the modified Laplacian coflow in \cite{Grigorian-2013}. In 2018, Chen provided the Shi-type estimate and studied finite-time singularities of this flow in \cite{Chen Shi-estimates}. Bedulli and Vezzoni showed the stability of the modified Laplacian coflow $\eqref{modified Laplacian coflow}$ in \cite{BV2020}. Recently, S\'a Earp, Saavedra, and Suan explored solutions to the modified Laplacian coflow $\eqref{modified Laplacian coflow}$ on a contact Calabi–Yau $7$-manifold in \cite{SSS2024}. For more details about the Laplacian coflow $\eqref{The coflow}$, the modified Laplacian coflow $\eqref{modified Laplacian coflow}$ and other flows related to $G_{2}$-structures, see \cite{Fer-fin 2016, Fino 2021, Grigorian-2016, Huang-wang-yao 2018,LLX25, Lotay 2022, Weiss-Witt 2012, Weiss-Witt 2012b}.

In \cite{Grigorian-2020}, Grigorian proposed the following question:
\begin{quote}
{\bf Question:} {\textit{Are the solutions of the modified Laplacian coflow real analytic, similarly to to Laplacian flow $\eqref{The closed Laplacian flow}$? } }
\end{quote}

In this paper, we study the real analyticity of the modified Laplacian coflow $\eqref{modified Laplacian coflow}$ with coclosed $G_{2}$-structures. Motivated by \cite{Bando 1987, L-W 2019 2, Kotschwar 2013}, we give the following estimate:
\begin{theorem}[{\bf Improvement of the Shi-type estimate}]
     Let $\psi(t)_{t\in[0,T]}$ be a smooth solution of the modified Laplacian coflow $\eqref{modified Laplacian coflow}$ on a geodesic ball $B_{g_{0}}(p,r)$, where $p\in M$, $r>0$, $M$ is the $7$-dimensional smooth manifold, and $g_{0}$ is the metric induced by initial coclosed $G_{2}$-structure $\psi$. Suppose that $A^{2}\leq \mathbf{M}$. If 
     $$\Lambda(x,t):=\left(|{\rm Rm}|^{2}(x,t)+|\nabla\mathbf{T}|^{2}(x,t)+|\mathbf{T}|^{4}(x,t)\right)^{\frac{1}{2}}\leq \mathbf{M}$$
     on $B_{g_{0}}(p,r)\times[0,T]$, then 
      there exist positive constants $L,C$ depending only on $r,\mathbf{M}$ and $T$, such that
     \begin{align}
         t^{\frac{k}{2}}\left(|\nabla^{k}{\rm Rm}|(x,t)+|\nabla^{k+1}\mathbf{T}|(x,t)\right)\leq CL^{\frac{k}{2}}(k+1)!
     \end{align}
     for all $k\in\mathbb{N}$ and $(x,t)\in B_{g_{0}}(p,r/2)\times[0,T]$.
\end{theorem}
\begin{remark}
In Theorem 2.1 of \cite{Chen Shi-estimates}, it is unclear how the upper bound $C(k)$ of the Shi-type estimate depends on the order $k$ of differentiation. Therefore, our estimate in Theorem 1.1 can be viewed as an improvement of the Shi-type estimate for the modified Laplacian coflow $\eqref{modified Laplacian coflow}$.
\end{remark}

Given a $7$-manifold $M$, a real analytic structure on $M$ is an atlas
$$\left\{(U_{j},\{x_{j}^{i}\}_{i={1}}^{7})\right\}_{j\in\mathcal{J}},$$
where $\mathcal{J}$ is some indexing set, such that the transition functions are real analytic. A Riemannian metric $g$ on a real analytic manifold $M$ is then \textit{real analytic} if the components $g_{ij}$ of $g$ are real analytic functions with respect to a subatlas of real analytic coordinates. If in addition the components $\psi_{ijkl}$ of $\psi$ are real analytic with respect to the normal coordinates, then we say that $\psi$ is also \textit{real analytic}.

Using Theorem 1.1, we provide real analyticity of the modified Laplacian coflow $\eqref{modified Laplacian coflow}$ with coclosed $G_{2}$-structures and answer the question proposed by Grigorian in \cite{Grigorian-2020}. 
\begin{theorem}[{\bf {Answer to} the question proposed by Grigorian}]\label{1.3}
    Let $\psi(t)_{t\in[0,T]}$ be a smooth solution of the modified Laplacian coflow $\eqref{modified Laplacian coflow}$ with coclosed $G_{2}$-structure on an open set $U\subset M$. For each time $t\in(0,T]$, $(U,\psi(t),g_{\psi}(t))$ is real analytic.
\end{theorem}

{Solitons are defined as self-similar solutions to the geometric flow. By Theorem \ref{1.3}, the solitons of the modified Laplacian coflow are consequently real analytic, a property that greatly facilitates their analysis. Further details on the study of these solitons can be found in \cite{Grigorian-2016}.}

We say a $G_{2}$-structure $\psi$ on $M$ is \textit{complete} if the associated metric $g_{\psi}$ is complete. Following Theorem 1.2 and Corollary 6.4 in \cite{KN1963}, we obtain the unique-continuation results for complete solutions.
\begin{corollary}[{\bf Unique-continuation results}]
    Let $M$ be a connected and simply connected $7$-manifold, and $\psi(t)$, $\tilde{\psi}(t)$ be smooth complete solutions to the modified Laplacian coflow $\eqref{modified Laplacian coflow}$ on $M\times[0,T]$. Then for any $t\in (0,T]$, the followings hold.
    \begin{itemize}
        \item[(i)] If $\psi(t)=\tilde{\psi}(t)$ on some connected open set $U\subset M$, then there exists a diffeomorphism $F$ of $M$ such that $F^{\ast}\tilde{\psi}(t)\equiv\psi(t)$.
        \item[(ii)] Any local diffeomorphism $F:U\rightarrow V$ between connected open sets $U,V\subset M$ satisfying $F^{\ast}(\psi(t)|_{V})=\psi(t)|_{U}$ can be uniquely extended to a global diffeomorphism $F$ of $M$ with $F^{\ast}(\psi(t))=\psi(t)$.
    \end{itemize}
\end{corollary}
\begin{remark}
    Since $\psi$ induces a unique metric $g_{\psi}$, any diffeomorphism $F:(M,\psi)\rightarrow(M,\tilde{\psi})$ such that $F^{\ast}\tilde{\psi}(t)=\psi(t)$ is an isometry between $(M,g_{\psi})$ and $(M,g_{\tilde{\psi}})$. However, $\psi$ and $g_{\psi}$ are not one-to-one, so the converse may {not be} true.
\end{remark}

We provide an outline of this paper. We review the basic theory in Section \ref{section2} about $G_{2}$-structure, $G_{2}$-decompositions of $2$-forms and $3$-forms, the torsion tensors of $G_{2}$-structures, and some definitions and conventions of metric. In Section 3, we derive some evolution equations under the modified Laplacian coflow $\eqref{modified Laplacian coflow}$. In Section 4, using derivative estimates for the Riemann curvature tensor ${\rm Rm}$, the torsion tensor $\mathbf{T}$, $\varphi$, and $\psi$ under modified Laplacian coflow $\eqref{modified Laplacian coflow}$, we give an improvement of Shi-type estimate. Finally, in Section 5, we utilize this estimate to establish the real analyticity of the modified Laplacian coflow $\eqref{modified Laplacian coflow}$ with coclosed $G_{2}$-structures.

{To orient the reader, we provide a roadmap as follows figuture \ref{roadmap}, which is the structure of the whole paper.}

\begin{figure}[!ht]\label{roadmap}
  \centering
  \hspace*{1pt}
  \begin{tikzpicture}[node distance = 3cm, auto]
    \node [block] (Sec2) {Section \ref{section2}};
    \node [block, right of = Sec2] (Sec3) {Section \ref{section3}};
    
    \node [right of = Sec3 ] (Sec) {};
    \node [block, below of = Sec ,node distance = 1.5cm] (Sec5) {Section 
    \ref{subsection5.1}};
    \node [block, below of = Sec ,node distance = -1.5cm] (Sec4) {Section \ref{section4}};
    \node [block, right of = Sec ] (Sec6) {Section \ref{subsection5.2}};

    \path [line] (Sec2) -- (Sec3);
    \path [line] (Sec2) -- (Sec4);
    \path [line] (Sec2) -- (Sec5);
    \path [line] (Sec3) -- (Sec4);
    \path [line] (Sec4) -- (Sec5);
    \path [line] (Sec4) -- (Sec6);
    \path [line] (Sec5) -- (Sec6);
  \end{tikzpicture}
 \caption{Roadmap}

\end{figure}

\section{$G_{2}$-structure}\label{section2}
In this section, we review some basis theory of $G_{2}$-structures in \cite{Bryant 2006, Spiros flow of G2,Spiros introduce to G2}.

\subsection{$G_{2}$-structure on smooth $7$-manifold}\label{subsection2.1}

Let $\{e_{1},e_{2},\cdots,e_{7}\}$ denote the standard basis of $\mathbb{R}^{7}$ and let $\{e^{1},e^{2},\cdots,e^{7}\}$ be its dual basis. Define the 3-form
$$\phi:=e^{123}+e^{145}+e^{167}+e^{246}-e^{257}-e^{347}-e^{356},$$
where $e^{ijk}:=e^{i}\wedge e^{j}\wedge e^{k}$. When we fix $\phi$, the subgroup of ${\rm GL}(7,\mathbb{R})$ is the exceptional Lie group $G_{2}$, which is a compact, connected, simple $14$-dimensional Lie subgroup of ${\rm SO}(7)$. In fact, $G_{2}$ acts irreducibly on $\mathbb{R}^{7}$ and preserves the metric and orientation for which $\{e_{1},e_{2},\cdots,e_{7}\}$ is an oriented orthonormal basis. Note that $G_{2}$ also preserves the $4$-form
$$\ast_{\phi}^{H}\phi=e^{4567}+e^{2367}+e^{2345}+e^{1357}-e^{1346}-e^{1256}-e^{1247},$$
where $\ast_{\phi}^{H}$ is the Hodge star operator determined by the metric and orientation.

For a smooth $7$-manifold $M$ and a point $x\in M$, we define as in \cite{local curvature of G2,Lotay-Wei Shi-estimate}
\begin{align}
    \wedge_{+}^{3}(T_{x}^{\ast}M):=\left\{\varphi_{x}\in\wedge^{3}(T_{x}^{\ast}M)\Big{|} u^{\ast}\phi=\varphi_{x},
    \text{for\ invertible}\ u\in\text{Hom}_{\mathbb{R}}(T_{x}^{\ast}M,\mathbb{R}^{7})\right\}\notag
\end{align}
and the bundle
\begin{align}
    \wedge_{+}^{3}(T^{\ast}M):=\bigcup_{x\in M}\wedge_{+}^{3}(T_{x}^{\ast}M)\notag.
\end{align}
We call a section $\varphi$ of $\wedge_{+}^{3}(T^{\ast}M)$ a {\it positive $3$-form} on $M$ or a {\it $G_{2}$-structure} on $M$, and denote the space of positive 3-form by $\Omega^{3}_{+}(M)$. The existence {of a $G_{2}$-structure} is equivalent to the
property that $M$ is oriented and spin, which is equivalent to the vanishing of the first two Stiefel-Whitney classes $\omega_{1}(TM)$ and $\omega_{2}(TM)$. For more details, see Theorem 10.6 in \cite{Lawson-Michelsohn}.

For a $3$-form $\varphi$, we define a $\Omega^{7}(M)$-valued bilinear form $\text{B}_{\varphi}$ by
$$\text{B}_{\varphi}(u,v)=\frac{1}{6}(u\lrcorner\varphi)\wedge(v\lrcorner\varphi)\wedge\varphi,$$
where $u, v$ are tangent vectors on $M$ and $``\lrcorner"$ is the interior multiplication operator (\textit{Here we use the orientation in \cite{Bryant 2006}}). Then we can see that any $\varphi\in\Omega^{3}_{+}(M)$ determines a Riemannian metric $g_{\varphi}$ and an orientation $d V_{\varphi}$, hence the Hodge star operator $\ast_{\varphi}^{H}$ and the associated $4$-form
$$\psi:=\ast_{\varphi}^{H}\varphi$$
can also be uniquely determined by $\varphi$. \textbf{Given the $4$-form $\psi$ and the orientation $\ast^{H}$, we can derive the $3$-form $\varphi$, thus we also {call} $\psi$ the $G_{2}$-structure.} If we choose local coordinates $\{x^{1},\cdots,x^{7}\}$ on $M$, then we can write $\varphi,\ \psi$ locally as
$$\varphi=\frac{1}{6}\varphi_{ijk}dx^{i}\wedge dx^{j}\wedge dx^{k},\ \psi=\frac{1}{24}\psi_{ijkl}dx^{i}\wedge dx^{j}\wedge dx^{k}\wedge dx^{l}.$$
The {relations} of $\varphi,\ \psi,\ g$ satisfy the following

\begin{lemma}[\cite{Bryant 2006,Spiros introduce to G2}]\label{lemma2.1}
    In local coordinates on $M$, the tensors $\varphi,\ \psi$ and $g$ satisfy the following relations:
    \begin{align}
        \varphi_{ijk}\varphi_{abc}g^{kc}&=g_{ia}g_{jb}-g_{ib}g_{ja}+\psi_{ijab},\\
        \varphi_{ijk}\varphi_{abc}g^{jb}g^{kc}&=6g_{ia},\\
        \psi_{ijkl}\psi_{abcd}g^{jb}g^{kc}g^{ld}&=24g_{ia},\\
        \varphi_{ijq}\psi_{abkl}g^{ia}g^{jb}&=4\varphi_{qkl}.
    \end{align}
\end{lemma}

\subsection{$G_{2}$-decomposition of $\Omega^{2}(M)$ and $\Omega^{3}(M)$}\label{subsection2.2} 

The group $G_{2}$ acts irreducibly on $\mathbb{R}^{7}$ (and hence on $\wedge^{1}(\mathbb{R}^{7})^{\ast}$ and $\wedge^{6}(\mathbb{R}^{7})^{\ast}$), but it
acts reducibly on $\wedge^{k}(\mathbb{R}^{7})^{\ast}$ for $2\leq k\leq 5$. Hence a $G_{2}$ structure $\varphi$ induces splittings
of the bundles $\wedge^{k}(T^{\ast}M)(2\leq k\leq5)$ into direct summands, which we denote by
$\wedge^{k}_{l}(T^{\ast}M,\varphi)$ with $l$ being the rank of the bundle. We let the space of sections
of $\wedge^{k}_{l}(T^{\ast}M,\varphi)$ be $\Omega^{k}_{l}(M)$. Define the natural projections
$$\pi^{k}_{l}:\Omega^{k}(M)\longrightarrow \Omega^{k}_{l}(M),\ \ \alpha\longmapsto \pi^{k}_{l}(\alpha).$$
Then we have
\begin{align}
    \Omega^{2}(M)&=\Omega^{2}_{7}(M)\oplus\Omega^{2}_{14}(M),\notag\\
    \Omega^{3}(M)&=\Omega^{3}_{1}(M)\oplus\Omega^{3}_{7}(M)\oplus\Omega^{3}_{27}(M)\notag.
\end{align}
where each component is determined by
\begin{align}
    \Omega^{2}_{7}(M)&=\{X\lrcorner\varphi:X\in C^{\infty}(TM)\}=\{\beta\in\Omega^{2}(M):\ast_{\varphi}^{H}(\varphi\wedge\beta)=2\beta\},\notag\\
    \Omega^{2}_{14}(M)&=\{\beta\in\Omega^{2}(M):\psi\wedge\beta=0\}=\{\beta\in\Omega^{2}(M):\ast_{\varphi}^{H}(\varphi\wedge\beta)=-\beta\},\notag
\end{align}
and
\begin{align}
    \Omega^{3}_{1}(M)&=\{f\varphi:f\in C^{\infty}(M)\},\notag\\
    \Omega^{3}_{7}(M)&=\{\ast_{\varphi}^{H}(\varphi\wedge\alpha):\alpha\in\Omega^{1}(M)\}=\{X\lrcorner\psi:X\in C^{\infty}(TM)\},\notag\\
    \Omega^{3}_{27}(M)&=\{\eta\in\Omega^{3}(M):\eta\wedge\varphi=\eta\wedge\psi=0\}.\notag
\end{align}

\begin{remark}\label{remark2.2}
    $\Omega^{4}$ and $\Omega^{5}$ have the corresponding decompositions by Hodge duality. The more details for $G_{2}$-decomposition see \cite{Bryant 2006,Spiros introduce to G2}.
\end{remark}

\subsection{The torsion tensors of $G_{2}$-structure}
\label{subsection2.3}

By the definition  of $G_{2}$ decomposition, we can find unique differential forms
$\tau_{0}\in\Omega^{0}(M),\tau_{1},\widetilde{\tau}_{1}\in\Omega^{1}(M),\tau_{2}\in\Omega^{2}_{14}(M)$ and $\tau_{3}\in\Omega^{3}_{27}(M)$ such that (see \cite{Bryant 2006})
\begin{align}
    d\varphi&=\tau_{0}\psi+3\!\ \tau_{1}\wedge\varphi+\ast_{\varphi}^{H}\tau_{3},\\
    d\psi&=4\!\ \widetilde{\tau}_{1}\wedge\psi+\tau_{2}\wedge\varphi.
\end{align}
In fact, {Bryant \cite{Bryant 2006}}  proved $\tau_{1}=\widetilde{\tau}_{1}$. We call $\tau_{0}$ the \textit{scalar torsion}, $\tau_{1}$ the \textit{vector torsion}, $\tau_{2}$ the \textit{Lie algebra torsion}, and $\tau_{3}$ the \textit{symmetric traceless torsion}. We also call $\tau_{\varphi}:=\{\tau_{0},\tau_{1},\tau_{2},\tau_{3}\}$ the intrinsic torsion forms of the $G_{2}$-structure $\varphi$.
From \cite{Spiros flow of G2,Spiros introduce to G2}, the full torsion tensor $\mathbf{T}=\mathbf{T}_{ij}dx^{i}\otimes dx^{j}$ satisfies the following:
\begin{align}
\label{torsion}
    \nabla_{i}\varphi_{jkl}&=\mathbf{T}_{i}^{\ m}\psi_{mjkl},\\
    \mathbf{T}_{i}^{\ j}&=\frac{1}{24}\nabla_{i}\varphi_{lmn}\psi^{jlmn},
\end{align}
and 
\begin{align}
\label{torsion of 4-form}
    \nabla_{m}\psi_{ijkl}=-(\mathbf{T}_{mi}\varphi_{jkl}-\mathbf{T}_{mj}\varphi_{ikl}-\mathbf{T}_{mk}\varphi_{jil}-\mathbf{T}_{ml}\varphi_{jki}),
\end{align}
{where $\nabla$ is the Levi-Civita connection of the metric that is determined by $\psi$.}
The full torsion tensor $\mathbf{T}_{ij}$ is related to the
intrinsic torsion forms by the following:
\begin{align}
\label{Def of T}\mathbf{T}_{ij}=\frac{\tau_{0}}{4}g_{ij}-(\tau_{3})_{ij}-(\tau_{1}^{\#}\lrcorner\varphi)_{ij}-\frac{1}{2}(\tau_{2})_{ij}
\end{align}
where $(\tau_{1}^{\#}\lrcorner\varphi)_{ij}=(\tau_{1}^{\#})^{l}\varphi_{lij}$ and $\#$ is the isomorphism from $1$-form to vector fields.

${}$

If $\varphi$ is coclosed, which means $d\psi=0$, then $\tau_{1}$ and $\tau_{2}$ are all zero from (2.6), so the nonzero torsion forms are
$\tau_{1}$ and $\tau_{3}$.
Thus according to Proposition 2.31 in \cite{Spiros flow of G2} and (2.10), we get
$$\mathbf{T}_{ij}=\frac{\tau_{0}}{4}g_{ij}-(\tau_{3})_{ij}$$
is a fully symmetric 2-tensor. {From Lemma $4.5$ in \cite{Grigorian-2013}, the torsion tensor $\mathbf{T}$ satisfies
$${\rm div}\ \mathbf{T}=\nabla\ {\rm tr}(\mathbf{T}).$$}

\subsection{Definitions and conventions of metric}\label{subsection2.4}
The Riemann curvature $(3,1)$-tensor field ${\rm Rm}$ of $g$ is defined by
$$
{\rm Rm}(X,Y)Z:=\nabla_{X}\nabla_{Y}Z-\nabla_{Y}\nabla_{X}Z-\nabla_{[X,Y]}Z,
$$
where $X,Y,Z$ are vector fields on $M$. In local coordinates $\{x^{1},\cdots,x^{7}\}$ on $M$, if we choose $\displaystyle{X=\frac{\partial}{\partial x^{i}}, Y=\frac{\partial}{\partial x^{j}}, Z=\frac{\partial}{\partial x^{k}}}$, $1\leq i,j,k\leq 7$, then we can write
$$
{\rm Rm}\left(\frac{\partial}{\partial x^{i}},\frac{\partial}{\partial x^{j}}\right)\frac{\partial}{\partial x^{k}}=R_{ijk}^{\ \ \ l}\frac{\partial}{\partial x^{l}},\ \ R_{ijkl}=g^{lm}R_{ijk}^{\ \ \ m}.
$$

We define the inner product for any two $k$-tensors $A=A_{i_{1}\cdots i_{k}}dx^{i_{1}}\otimes\cdots\otimes dx^{i_{k}},B=B_{i_{1}\cdots i_{k}}dx^{i_{1}}\otimes\cdots\otimes dx^{i_{k}}$ as
$$
\langle A,B\rangle=A_{i_{1}\cdots i_{k}}B^{i_{1}\cdots i_{k}},
$$
then the norm of any $k$-tensor $A=A_{i_{1}\cdots i_{k}}dx^{i_{1}}\otimes\cdots\otimes dx^{i_{k}}$ is
\begin{align}\label{norm}
|A|^{2}=\langle A,A\rangle=A_{i_{1}\cdots i_{k}}A^{i_{1}\cdots i_{k}},
\end{align}
We denote 
$$\triangle_{g(t)}=g^{ij}\nabla_{i}\nabla_{j}$$
the trace Laplacian. For convenience, we always omit
the time variable $t$.

\section{Modified Laplacian coflow and evolution equations}\label{section3}
In this section, we give some evolution equations of the modified Laplacian coflow $\eqref{modified Laplacian coflow}$, which will be used in Section \ref{section4}.

In \cite{Spiros flow of G2}, the evolution equation of $\psi$ under the general flow can be written as the following:
\begin{align}
        \frac{\partial}{\partial t}\psi_{njkl}&=g^{pq}(h_{np}\psi_{qjkl}+h_{jp}\psi_{nqkl}+h_{kp}\psi_{njql}+h_{lp}\psi_{njkq})\notag\\
        &\quad-X_{n}\varphi_{jkl}+X_{j}\varphi_{nkl}-X_{k}\varphi_{njl}+X_{l}\varphi_{njk}\notag,
\end{align}
where $h$ is the symmetric $2$-tensor and $X$ is the vector field. When {closed} $G_{2}$-structure $\psi$ is evolved by  modified Laplacian coflow $\eqref{modified Laplacian coflow}$, {the full torsion tensor $\mathbf{T}_{ij}$ is symmetric,} from \cite{Chen Shi-estimates,Grigorian-2013}, the symmetric $2$-tensor and vector field $X$ are
\begin{align}
h_{ij}=-R_{ij}+\frac{1}{2}\mathbf{T}^{km}\mathbf{T}^{ln}\varphi_{ikl}\varphi_{jmn}+(2A-{\rm tr}(\mathbf{T}))\mathbf{T}_{ij},\quad X=\nabla{\rm tr}(\mathbf{T}).
\end{align}
{Thus,} the modified Laplacian coflow can be written as
\begin{equation}
\left\{
\begin{array}{rl}
\frac{\partial}{\partial t}g_{ij}&=-2R_{ij}+\mathbf{T}^{km}\mathbf{T}^{ln}\varphi_{ikl}\varphi_{jmn}+(4A-2{\rm tr}(\mathbf{T}))\mathbf{T}_{ij},\\
X&=\nabla{\rm tr}(\mathbf{T}).\\
\end{array}
\right.
\label{Ricci-like coflow}
\end{equation}

Next we calculate some evolution equations under the modified Laplacian coflow $\eqref{modified Laplacian coflow}$.
\begin{lemma}\label{lemma 3.1}
    Under the modified Laplacian coflow $\eqref{modified Laplacian coflow}$, the full torsion $\mathbf{T}$ {satisfies the following evolution equation:}
    \begin{align}
       \left(\frac{\partial}{\partial t}-\triangle\right)\mathbf{T}&={\rm Rm}\ast\mathbf{T}+A\ast\nabla\mathbf{T}\ast\varphi+\nabla\mathbf{T}\ast\mathbf{T}\ast\varphi+\nabla\mathbf{T}\ast\mathbf{T}\ast\varphi\ast\varphi\ast\varphi\notag\\ &\quad+\mathbf{T}\ast\nabla\varphi\ast\nabla\varphi+\mathbf{T}\ast\mathbf{T}\ast\mathbf{T}+A\ast\mathbf{T}\ast\mathbf{T}+\mathbf{T}\ast\mathbf{T}\ast\mathbf{T}\ast\varphi\ast\varphi\notag.
    \end{align}   
\end{lemma}
\begin{proof}
    From ${\rm Ric}=7\ast{\rm Rm}$, ${\rm tr}(\mathbf{T})=7\ast\mathbf{T}$ and $\eqref{torsion}$, this evolution equation can be derived from Proposition 5.2 in \cite{Grigorian-2013}.
\end{proof}

\begin{lemma}\label{lemma 3.2}
    Under the modified Laplacian coflow $\eqref{modified Laplacian coflow}$, the Riemann curvature tensor ${\rm Rm}$ {satisfies the following evolution equation:}
    \begin{align}
        \left(\frac{\partial}{\partial t}-\triangle\right){\rm Rm}&={\rm Rm}\ast{\rm Rm}+{\rm Rm}\ast\mathbf{T}\ast\mathbf{T}\ast\varphi\ast\varphi+A\ast{\rm Rm}\ast\mathbf{T}+{\rm Rm}\ast\mathbf{T}\ast\mathbf{T}\notag\\
        &\quad+\nabla^{2}(\mathbf{T}\ast\mathbf{T}\ast\varphi\ast\varphi)+A\ast\nabla^{2}\mathbf{T}+\nabla^{2}(\mathbf{T}\ast\mathbf{T})\notag.
    \end{align}
\end{lemma}
\begin{proof}
    For the geometric flow 
\begin{align}
\partial_{t}g_{ij}=2h_{ij},
\label{geometric flow}
\end{align}
where $h_{ij}$ is a family of symmetric $2$-tensor. From Lemma 6.5 in \cite{Chow-Knopf 2004}, we have the following evolution equation for {the} Riemann curvature tensor:
\begin{equation}
\frac{\partial}{\partial t}R_{ijk}^{\ \ \ l}=g^{pl}\left\{
\begin{array}{c}
\nabla_{i}\nabla_{k}h_{jp}+\nabla_{j}\nabla_{p}h_{ik} -\nabla_{i}\nabla_{p}h_{jk} -\nabla_{j}\nabla_{k}h_{ip}\\
-R_{ijk}^{\ \ \ q}h_{pq} -R_{ijp}^{\ \ \ q}h_{kq}
\end{array}
\right\}.
\end{equation}
Under the modified Laplacian coflow $\eqref{Ricci-like coflow}$, the $2$-tensor $h$ should be
$$h_{ij}=-R_{ij}+\frac{1}{2}\mathbf{T}^{km}\mathbf{T}^{ln}\varphi_{ikl}\varphi_{jmn}+(2A-{\rm tr}\mathbf{T})\mathbf{T}_{ij}.$$
Thus, the evolution equation of Riemann curvature ${\rm Rm}$ should be
\begin{align}
    \frac{\partial}{\partial t}R_{ijkl}&=\frac{\partial}{\partial t}\left(g_{pl}R_{ijk}^{\ \ \ p}\right)=\left(\frac{\partial}{\partial t}g_{pl}\right)\cdot R_{ijk}^{\ \ \ p}+g_{pl}\cdot \left(\frac{\partial}{\partial t}R_{ijk}^{\ \ \ p}\right)\\   &=\left[-2R_{pl}+\mathbf{T}^{am}\mathbf{T}^{bn}\varphi_{pal}\varphi_{lbn}+(4A-2{\rm tr}(\mathbf{T}))\mathbf{T}_{pl}\right]R_{ijk}^{\ \ \ p}\notag\\
    &\quad+g_{pl}\cdot \left(\frac{\partial}{\partial t}R_{ijk}^{\ \ \ p}\right)\notag.
\end{align}
For the second term, we have
\begin{align}
    &\quad\nabla_{i}\nabla_{k}h_{jl}+\nabla_{j}\nabla_{l}h_{ik} -\nabla_{i}\nabla_{l}h_{jk} -\nabla_{j}\nabla_{k}h_{il}\notag\\
    &=-\nabla_{i}\nabla_{k}R_{jl}-\nabla_{j}\nabla_{l}R_{ik} +\nabla_{i}\nabla_{l}R_{jk} +\nabla_{j}\nabla_{k}R_{il}\notag\\
    &\quad+\frac{1}{2}\nabla_{i}\nabla_{k}(\mathbf{T}^{pm}\mathbf{T}^{qn}\varphi_{jpq}\varphi_{lmn})+\frac{1}{2}\nabla_{j}\nabla_{l}(\mathbf{T}^{pm}\mathbf{T}^{qn}\varphi_{ipq}\varphi_{kmn})\notag\\
    &\quad-\frac{1}{2}\nabla_{i}\nabla_{l}(\mathbf{T}^{pm}\mathbf{T}^{qn}\varphi_{jpq}\varphi_{kmn})-\frac{1}{2}\nabla_{j}\nabla_{k}(\mathbf{T}^{pm}\mathbf{T}^{qn}\varphi_{ipq}\varphi_{lmn})\notag\\  &\quad+2A(\nabla_{i}\nabla_{k}\mathbf{T}_{jl}+\nabla_{j}\nabla_{l}\mathbf{T}_{ik} -\nabla_{i}\nabla_{l}\mathbf{T}_{jk}-\nabla_{j}\nabla_{k}\mathbf{T}_{il})\notag\\
    &\quad-\left[\nabla_{i}\nabla_{k}({\rm tr}(\mathbf{T})\mathbf{T}_{jl})+\nabla_{j}\nabla_{l}({\rm tr}(\mathbf{T})\mathbf{T}_{ik}) -\nabla_{i}\nabla_{l}({\rm tr}(\mathbf{T})\mathbf{T}_{jk})-\nabla_{j}\nabla_{k}({\rm tr}(\mathbf{T})\mathbf{T}_{il})\right]\notag.
\end{align}
As in Ricci flow, {using} the Bianchi {identities} and commuting covariant derivatives, one gets
\begin{align}
    &\quad-\nabla_{i}\nabla_{k}R_{jl}-\nabla_{j}\nabla_{l}R_{ik}+\nabla_{i}\nabla_{l}R_{jk} +\nabla_{j}\nabla_{k}R_{il}\notag\\
    &=\triangle R_{ijkl}+R_{ijpr}R^{rp}_{\ \ kl}-2R^{p\ \ r}_{\ ik}R_{jprl}+2R^{p}_{\ irl}R_{jpk}^{\ \ \ r}-g^{pq}(R_{ip}R_{qjkl}+R_{jp}R_{iqkl})\notag,
\end{align}
and
\begin{align}
    &\quad-R_{ijk}^{\ \ \ p}h_{pl}-R_{ijl}^{\ \ \ p}h_{kp}\notag\\
    &=R_{ijk}^{\ \ \ p}R_{pl}+R_{ijl}^{\ \ \ p}R_{kp}-\frac{1}{2}R_{ijk}^{\ \ \ p}\mathbf{T}^{am}\mathbf{T}^{bn}\varphi_{qab}\varphi_{lmn}-\frac{1}{2}R_{ijl}^{\ \ \ p}\mathbf{T}^{am}\mathbf{T}^{bn}\varphi_{qab}\varphi_{kmn}\notag\\
    &\quad-2A(R_{ijk}^{\ \ \ p}\mathbf{T}_{pl}+R_{ijl}^{\ \ \ p}\mathbf{T}_{kp})+{\rm tr}\mathbf{T}(R_{ijk}^{\ \ \ p}\mathbf{T}_{pl}+R_{ijl}^{\ \ \ p}\mathbf{T}_{kp})\notag.
\end{align}
Together with the above, and note that ${\rm Ric}=7\ast{\rm Rm}$, ${\rm tr}(\mathbf{T})=7\ast\mathbf{T}$, we get the desired evolution equation.
\end{proof}

\begin{lemma}\label{lemma 3.3}
    Under the modified Laplacian coflow $\eqref{modified Laplacian coflow}$, $\nabla\varphi$ {satisfies the following evolution equation:} 
    \begin{align}
        \left(\frac{\partial}{\partial t}-\triangle\right)\nabla\varphi&={\rm Rm}\ast\mathbf{T}\ast\psi+\mathbf{T}\ast\nabla\varphi\ast\nabla\varphi\ast\psi+A\ast\mathbf{T}\ast\mathbf{T}\ast\psi\notag\\  &\quad+\nabla\mathbf{T}\ast(\mathbf{T}\ast\varphi+A\ast\varphi\ast\psi+\mathbf{T}\ast\varphi\ast\psi+\mathbf{T}\ast\varphi\ast\varphi\ast\varphi\ast\psi)\notag\\
        &\quad+\mathbf{T}\ast\mathbf{T}\ast\mathbf{T}\ast(\psi+\varphi\ast\varphi\ast\psi)\notag.
    \end{align}
\end{lemma}
\begin{proof}
    From $\eqref{torsion}$, we can calculate the evolution equation of $\nabla\varphi$ as follows:
    \begin{align}
    \label{3.8}
        \frac{\partial}{\partial t}\nabla_{i}\varphi_{jkl}&=\frac{\partial}{\partial t}\left(g^{mn}\mathbf{T}_{im}\psi_{njkl}\right)\\
        &=\left(\frac{\partial}{\partial t}g^{mn}\right)\mathbf{T}_{im}\psi_{njkl}+g^{mn}\left(\frac{\partial}{\partial t}\mathbf{T}_{im}\right)\psi_{njkl}+g^{mn}\mathbf{T}_{im}\left(\frac{\partial}{\partial t}\psi_{njkl}\right)\notag\\
        &={\rm Rm}\ast\mathbf{T}\ast\psi+\mathbf{T}\ast\mathbf{T}\ast\mathbf{T}\ast\varphi\ast\varphi\ast\psi+A\ast\mathbf{T}\ast\mathbf{T}\ast\psi\notag\\
        &\quad+\mathbf{T}\ast\mathbf{T}\ast\mathbf{T}\ast\psi+g^{mn}\left(\frac{\partial}{\partial t}\mathbf{T}_{im}\right)\psi_{njkl}+g^{mn}\mathbf{T}_{im}\left(\frac{\partial}{\partial t}\psi_{njkl}\right)\notag,
    \end{align}
    and
    \begin{align}
       \triangle(\nabla_{i}\varphi_{jkl})&=g^{pq}g^{mn}\nabla_{p}\nabla_{q}(\mathbf{T}_{im}\psi_{njkl})\\
       &=g^{pq}g^{mn}\nabla_{p}[\left(\nabla_{q}\mathbf{T}_{im}\right)\cdot\psi_{njkl}+\mathbf{T}_{im}(\nabla_{q}\psi_{njkl})]\notag\\
&=g^{mn}\triangle\mathbf{T}_{im}\cdot\psi_{njkl}+\nabla\mathbf{T}\ast\nabla\psi+\mathbf{T}\ast\nabla\mathbf{T}\ast\varphi+\mathbf{T}\ast\mathbf{T}\ast\nabla\varphi\notag.
    \end{align}
    For the last term in $\eqref{3.8}$, we need to compute
    \begin{align}
        \frac{\partial}{\partial t}\psi_{njkl}&=g^{pq}(h_{np}\psi_{qjkl}+h_{jp}\psi_{nqkl}+h_{kp}\psi_{njql}+h_{lp}\psi_{njkq})\notag\\
        &\quad-X_{n}\varphi_{jkl}+X_{j}\varphi_{nkl}-X_{k}\varphi_{njl}+X_{l}\varphi_{njk}\notag,
    \end{align}
    where $h$ and $X$ are defined in $\eqref{Ricci-like coflow}$. Then we will calculate $g^{pq}h_{np}\psi_{qjkl}-X_{n}\varphi_{jkl}$ as a sample(others are similar to it):
    \begin{align}
        &\quad g^{pq}h_{np}\psi_{qjkl}-X_{n}\varphi_{jkl}\notag\\
        &=g^{pq}\left(-R_{ij}+\frac{1}{2}\mathbf{T}^{km}\mathbf{T}^{ln}\varphi_{ikl}\varphi_{jmn}+(2A-{\rm tr}(\mathbf{T}))\mathbf{T}_{ij}\right)\psi_{qjkl}+\nabla_{n}{\rm tr}(\mathbf{T})\cdot\varphi_{jkl}\notag\\
        &={\rm Rm}\ast\psi+\mathbf{T}\ast\mathbf{T}\ast\varphi\ast\varphi\ast\psi+A\ast\mathbf{T}\ast\psi+\mathbf{T}\ast\mathbf{T}\ast\psi+\nabla\mathbf{T}\ast\varphi\notag.
    \end{align}
    Applying $\nabla\varphi=\mathbf{T}\ast\psi,\ \nabla\psi=\mathbf{T}\ast\varphi$ to the above formulas, and combining with Lemma \ref{lemma 3.1}, we get the evolution equation of $\nabla\varphi$.
\end{proof}

\begin{lemma}
    Under the modified Laplacian coflow $\eqref{modified Laplacian coflow}$, $\nabla\psi$ {satisfies the following evolution equation:}
    \begin{align}
        \left(\frac{\partial}{\partial t}-\triangle\right)\nabla\psi&={\rm Rm}\ast\mathbf{T}\ast\varphi+\mathbf{T}\ast\nabla\varphi\ast\nabla\varphi\ast\varphi+A\ast\mathbf{T}\ast\mathbf{T}\ast\varphi\notag\\  &\quad+\nabla\mathbf{T}\ast(\mathbf{T}\ast\psi+A\ast\varphi\ast\varphi+\mathbf{T}\ast\varphi\ast\varphi+\mathbf{T}\ast\varphi\ast\varphi\ast\varphi\ast\varphi)\notag\\
        &\quad+\mathbf{T}\ast\mathbf{T}\ast\mathbf{T}\ast(\varphi+\varphi\ast\varphi\ast\varphi)\notag.
    \end{align}
\end{lemma}
\begin{proof}
    From $\eqref{torsion of 4-form}$, we can calculate the evolution equation of $\nabla\psi$ as following:
    \begin{align}
        \frac{\partial}{\partial t}\nabla_{m}\psi_{ijkl}&=-\frac{\partial}{\partial t}\mathbf{T}_{mi}\cdot\varphi_{jkl}+\frac{\partial}{\partial t}\mathbf{T}_{mj}\cdot\varphi_{ikl}+\frac{\partial}{\partial t}\mathbf{T}_{mk}\cdot\varphi_{jil}+\frac{\partial}{\partial t}\mathbf{T}_{ml}\cdot\varphi_{jki}\notag\\
        &\quad-\mathbf{T}_{mi}\cdot\frac{\partial}{\partial t}\varphi_{jkl}+\mathbf{T}_{mj}\cdot\frac{\partial}{\partial t}\varphi_{ikl}+\mathbf{T}_{mk}\cdot\frac{\partial}{\partial t}\varphi_{jil}+\mathbf{T}_{ml}\cdot\frac{\partial}{\partial t}\varphi_{jki}\notag.
    \end{align}
    Here we only calculate $\left(\frac{\partial}{\partial t}\mathbf{T}_{mi}\right)\cdot\varphi_{jkl}+\mathbf{T}_{mi}\left(\frac{\partial}{\partial t}\varphi_{jkl}\right)$ as {an example}(others are similar to it):
    \begin{align}
      \triangle(\mathbf{T}_{mi}\varphi_{jkl})&=g^{pq}\nabla_{p}\nabla_{q}(\mathbf{T}_{mi}\varphi_{jkl})\notag\\
      &=g^{pq}\nabla_{p}(\nabla_{q}\mathbf{T}_{mi}\cdot\varphi_{jkl}+\mathbf{T}_{mi}\nabla_{q}\varphi_{jkl})\notag\\
      &=\left(\triangle \mathbf{T}_{mi}\right)\cdot\varphi_{jkl}+\nabla\mathbf{T}\ast\nabla\varphi+ \mathbf{T}\ast\nabla^{2}\varphi\notag\\
      &=\left(\triangle \mathbf{T}_{mi}\right)\cdot\varphi_{jkl}+\nabla\mathbf{T}\ast\mathbf{T}\ast\psi+\mathbf{T}\ast\mathbf{T}\ast\nabla\psi.\notag
    \end{align}
    From \cite{Spiros flow of G2} and $\eqref{Ricci-like coflow}$, one gets 
    \begin{align}
        \frac{\partial}{\partial t}\varphi_{ijk}&=h_{jn}\varphi^{n}_{\ kl}-h_{kn}\varphi^{n}_{\ jl}+h_{ln}\varphi^{n}_{\ jk}+X^{n}\psi_{njkl}\notag\\
        &={\rm Rm}\ast\varphi+\mathbf{T}\ast\mathbf{T}\ast\varphi\ast\varphi\ast\varphi+A\ast\mathbf{T}\ast\varphi+\mathbf{T}\ast\mathbf{T}\ast\varphi+\nabla\mathbf{T}\ast\psi\notag.
    \end{align}
   Similarly, applying $\nabla\varphi=\mathbf{T}\ast\psi,\ \nabla\psi=\mathbf{T}\ast\varphi$ to the above formulas, and combining with Lemma \ref{lemma 3.1}, we get the evolution equation of $\nabla\psi$.
\end{proof}

\begin{lemma}\label{lemma 3.5}
    Under the modified Laplacian coflow $\eqref{modified Laplacian coflow}$, the evolution equation for the norm of full tensor $|\mathbf{T}|^{2}$ can be estimated as 
    \begin{align}
    \left(\frac{\partial}{\partial t}-\triangle\right)|\mathbf{T}|^{2}\leq C|{\rm Rm}||\mathbf{T}|^{2}+CA^{2}|\mathbf{T}|^{2}+CA|\mathbf{T}|^{3}+C|\mathbf{T}|^{4}.
    \end{align}
\end{lemma}
\begin{proof}
    According to Lemma \ref{lemma 3.1}, we have 
    \begin{align}
      \left(\frac{\partial}{\partial t}-\triangle\right)|\mathbf{T}|^{2}&=2\left(\frac{\partial}{\partial t}-\triangle\right)\mathbf{T}_{ij}\cdot\mathbf{T}^{ij}-2|\nabla\mathbf{T}|^{2}+2\left(\frac{\partial}{\partial t}g^{ij}\right)\mathbf{T}_{im}\mathbf{T}_{j}^{ \ m}\notag\\
      &=-2|\nabla\mathbf{T}|^{2}+{\rm Rm}\ast\mathbf{T}\ast\mathbf{T}+A\ast\nabla\mathbf{T}\ast\mathbf{T}\ast\varphi+\nabla\mathbf{T}\ast\mathbf{T}\ast\mathbf{T}\ast\varphi\notag\\
      &\quad+\nabla\mathbf{T}\ast\mathbf{T}\ast\mathbf{T}\ast\varphi\ast\varphi\ast\varphi+\mathbf{T}\ast\mathbf{T}\ast\nabla\varphi\ast\nabla\varphi\notag\\ &\quad+\mathbf{T}\ast\mathbf{T}\ast\mathbf{T}\ast\mathbf{T}+A\ast\mathbf{T}\ast\mathbf{T}\ast\mathbf{T}+\mathbf{T}\ast\mathbf{T}\ast\mathbf{T}\ast\mathbf{T}\ast\varphi\ast\varphi\notag\\
      &\leq -2|\nabla\mathbf{T}|^{2}+C|\nabla\mathbf{T}|\left(A|\mathbf{T}|+|\mathbf{T}|^{2}\right)+C|{\rm Rm}||\mathbf{T}|^{2}\notag\\
      &\quad+CA|\mathbf{T}|^{3}+C|\mathbf{T}|^{4}\notag\\
      &\leq C|{\rm Rm}||\mathbf{T}|^{2}+CA^{2}|\mathbf{T}|^{2}+CA|\mathbf{T}|^{3}+C|\mathbf{T}|^{4}\notag.
    \end{align}
    where the last inequality {uses the} Cauchy-Schwarz inequality. 
\end{proof}

\begin{lemma}\label{lemma 3.6}
    There exists a constant $C$, such that for any $p$-tensor $S$ with $p\geq 1$ and any integer $k$, we have the estimate
    \begin{align}
        |\nabla^{k}\triangle S-\triangle\nabla^{k}S|\leq C(p+1)\sum_{i=0}^{k}\frac{(k+2)!}{(i+2)!(k-i)!}|\nabla^{i}{\rm Rm}||\nabla^{k-i}S|.
    \end{align}
\end{lemma}
\begin{proof}
    This Lemma is Lemma 2 in \cite{Bando 1987}.
\end{proof}

\begin{lemma}\label{lemma 3.7}
    Under the modified Laplacian coflow $\eqref{modified Laplacian coflow}$, there exists a constant $C$, such that for any $p$-tensor $S$ with $p\geq 1$ and any integer $k$, we have the estimate 
    \begin{align}
        \left|\nabla^{k}\frac{\partial}{\partial t}S-\frac{\partial}{\partial t}\nabla^{k}S\right|&\leq C(p+1)\sum_{i=1}^{k}\frac{(k+1)!}{(i+1)!(k-i)!}|\nabla^{i}{\rm Rm}||\nabla^{k-i}S|\notag\\
        &\quad+C(p+1)\sum_{i=1}^{k}\frac{(k+1)!}{(i+1)!(k-i)!}|\nabla^{i}(\mathbf{T}\ast\mathbf{T}\ast\varphi\ast\varphi)||\nabla^{k-i}S|\notag\\
        &\quad+C(p+1)\sum_{i=1}^{k}\frac{(k+1)!}{(i+1)!(k-i)!}A|\nabla^{i}\mathbf{T}||\nabla^{k-i}S|\notag\\
        &\quad+C(p+1)\sum_{i=1}^{k}\frac{(k+1)!}{(i+1)!(k-i)!}|\nabla^{i}(\mathbf{T}\ast\mathbf{T})||\nabla^{k-i}S|\notag.
    \end{align}
\end{lemma}
\begin{proof}
    As the same discussion of Lemma 13.26 in \cite{Ricci flow TA 2}, for the geometric flow $\eqref{geometric flow}$, the commutator formula for $\nabla^{k}$ and $\frac{\partial}{\partial t}$ should be
    \begin{align}
    \label{3.14}
        \left|\nabla^{k}\frac{\partial}{\partial t}S-\frac{\partial}{\partial t}\nabla^{k}S\right|&\leq C(p+1)\sum_{i=1}^{k}\frac{(k+1)!}{(i+1)!(k-i)!}|\nabla^{i}h||\nabla^{k-i}S|.
    \end{align}
    Under the modified Laplacian coflow $\eqref{Ricci-like coflow}$, where
    $$h_{ij}=-R_{ij}+\frac{1}{2}\mathbf{T}^{km}\mathbf{T}^{ln}\varphi_{ikl}\varphi_{jmn}+(2A-{\rm tr}(\mathbf{T}))\mathbf{T}_{ij},$$
    by substituting it {into} $\eqref{3.14}$ and noting that $|{\rm Ric}|\leq 7|{\rm Rm}|$ holds,  we get the estimate of the commutator formula for $\nabla^{k}$ and $\frac{\partial}{\partial t}$  under the modified Laplacian coflow.
\end{proof}

Combining Lemmas \ref{lemma 3.6} and \ref{lemma 3.7}, one gets the following estimate:
\begin{prop}\label{prop 3.8}
    Under the modified Laplacian coflow $\eqref{modified Laplacian coflow}$, there exists a constant $C$, such that for any $p$-tensor $S$ with $p\geq 1$ and any integer $k$, we have the estimate 
    \begin{align}
        &\quad\left|\left(\frac{\partial}{\partial t}-\triangle\right)\nabla^{k}S-\nabla^{k}\left(\frac{\partial}{\partial t}-\triangle\right)S\right|\\
        &\leq C(p+1)\sum_{i=0}^{k}\frac{(k+2)!}{(i+2)!(k-i)!}|\nabla^{i}{\rm Rm}||\nabla^{k-i}S|\notag\\
        &\quad+C(p+1)\sum_{i=1}^{k}\frac{(k+1)!}{(i+1)!(k-i)!}|\nabla^{i}(\mathbf{T}\ast\mathbf{T}\ast\varphi\ast\varphi)||\nabla^{k-i}S|\notag\\
        &\quad+C(p+1)\sum_{i=1}^{k}\frac{(k+1)!}{(i+1)!(k-i)!}A|\nabla^{i}\mathbf{T}||\nabla^{k-i}S|\notag\\
        &\quad+C(p+1)\sum_{i=1}^{k}\frac{(k+1)!}{(i+1)!(k-i)!}|\nabla^{i}(\mathbf{T}\ast\mathbf{T})||\nabla^{k-i}S|\notag.
    \end{align}
\end{prop}

\begin{remark}
   The uniform constant $C$, appearing above, depends only on  the dimension $7$ of $M$.
\end{remark}

\section{Main derivative estimate}\label{section4}

In this section, we will present the main derivative estimate, which we also refer to as the improvement of the Shi-type estimate for the modified Laplacian coflow \eqref{modified Laplacian coflow}. Before deriving it, we introduce the following notions for the sake of convenience.

For $k\geq 0$, we define
\begin{align}
    a_{k}=\frac{t^{\frac{k}{2}}|\nabla^{k}{\rm Rm}|}{(k+1)!},b_{k}=\frac{t^{\frac{k}{2}}|\nabla^{k+1}\mathbf{T}|}{(k+1)!},c_{k}=\frac{t^{\frac{k}{2}}|\nabla^{k+2}\varphi|}{(k+1)!},d_{k}=\frac{t^{\frac{k}{2}}|\nabla^{k+2}\psi|}{(k+1)!}.
\end{align}
For $k\geq 1$, we define
\begin{align}\label{def 1.1}
    \tilde{a}_{k}=\frac{t^{\frac{k-1}{2}}|\nabla^{k}{\rm Rm}|}{k!},\tilde{b}_{k}=\frac{t^{\frac{k-1}{2}}|\nabla^{k+1}\mathbf{T}|}{k!},\tilde{c}_{k}=\frac{t^{\frac{k-1}{2}}|\nabla^{k+2}\varphi|}{k!},\tilde{d}_{k}=\frac{t^{\frac{k-1}{2}}|\nabla^{k+2}\psi|}{k!}.
\end{align}
By setting $k!=1$ for all $k\leq 0$, we define
\begin{align}
    \tilde{a}_{0}&=t^{-\frac{1}{2}}|{\rm Rm}|, \quad \tilde{b}_{0}=t^{-\frac{1}{2}}|\nabla\mathbf{T}|,\ \quad\tilde{c}_{0}=t^{-\frac{1}{2}}|\nabla^{2}\varphi|,\quad\tilde{d}_{0}=t^{-\frac{1}{2}}|\nabla^{2}\psi|\notag,\\
    b_{-1}&=t^{-\frac{1}{2}}|\mathbf{T}|,  \quad c_{-1}=t^{-\frac{1}{2}}|\nabla\varphi|,\quad d_{-1}=t^{-\frac{1}{2}}|\nabla\psi|,\notag\\
    \tilde{b}_{-1}&=t^{-1}|\mathbf{T}|,  \quad \tilde{c}_{-1}=t^{-1}|\nabla\varphi|,\quad \ \tilde{d}_{-1}=t^{-1}|\nabla\psi|,\notag\\
     c_{-2}&=t^{-1}|\varphi|,\  \quad d_{-2}=t^{-1}|\psi|,\ \ \  \quad \tilde{c}_{-2}=t^{-\frac{3}{2}}|\varphi|,\quad\ \ \tilde{d}_{-2}=t^{-\frac{3}{2}}|\psi|\notag.
\end{align}

Next, we define
\begin{align}
    \mathbf{A}_{N}=\sum_{k=0}^{N}a_{k}^{2},\quad \mathbf{B}_{N}=\sum_{k=0}^{N}b_{k}^{2},\quad \mathbf{C}_{N}=\sum_{k=0}^{N}c_{k}^{2},\quad \mathbf{D}_{N}=\sum_{k=0}^{N}d_{k}^{2},\\
    \tilde{\mathbf{A}}_{N}=\sum_{k=1}^{N}\tilde{a}_{k}^{2},\quad \tilde{\mathbf{B}}_{N}=\sum_{k=1}^{N}\tilde{b}_{k}^{2},\quad \tilde{\mathbf{C}}_{N}=\sum_{k=1}^{N}\tilde{c}_{k}^{2},\quad \tilde{\mathbf{D}}_{N}=\sum_{k=1}^{N}{\tilde{d}}_{k}^{2},
\end{align}
and 
\begin{align} \Phi_{N}&=\mathbf{A}_{N}+\mathbf{B}_{N}+\mathbf{C}_{N}+\mathbf{D}_{N}+|\mathbf{T}|^{2}+A^{2}+|\varphi|^{2}+|\psi|^{2},\notag\\
\Psi_{N}&=\tilde{\mathbf{A}}_{N}+\tilde{\mathbf{B}}_{N}+\tilde{\mathbf{C}}_{N}+\tilde{\mathbf{D}}_{N}\notag.
\end{align}

\begin{remark}
    Here $\Phi_{N}$ is {different from} Lotay and Wei's paper \cite{L-W 2019 2}, we add the norm of torsion $|\mathbf{T}|^{2}$ and some constants, where $|\varphi|^{2}=42,|\psi|^{2}=168$, {the norm $|\cdot|$ is defined in $\eqref{norm}$.}
\end{remark}

\begin{lemma}
    Let $(M,\psi(t))_{t\in[0,T]}$ be a solution of the modified Laplacian coflow $\eqref{modified Laplacian coflow}$. For any nonnegative integers $x,y,z,w$, we define
    \begin{align}
    P(x,y,z,w)&=t^{\frac{x+y+2z+2w}{2}}\left(\mathbf{B}_{N}+b_{-1}^{2}\right)^{\frac{x}{2}}\left(\mathbf{C}_{N}+c_{-1}^{2}\right)^{\frac{y}{2}}\left(\mathbf{C}_{N}+c_{-1}^{2}+c_{-2}^{2}\right)^{\frac{z}{2}}\\
    &\quad\cdot\left(\mathbf{D}_{N}+d_{-1}^{2}+d_{-2}^{2}\right)^{\frac{w}{2}}
    -t^{\frac{x+y+2z+2w}{2}}b_{-1}^{x}c_{-1}^{y}c_{-2}^{z}d_{-2}^{w},\notag
    \end{align}
    then we get 
    $$P(x,y,z,w)\leq C\sum_{i=1}^{x+y+2z+2w}t^{\frac{i}{2}}\Phi_{N}^{\frac{x+y+z+w}{2}},$$
    where $C$ is a positive constant depends only on $x,y,z,w$.
\end{lemma}
\begin{proof}
    We can obtain this lemma by {a direct calculation.}
\end{proof}

\begin{lemma}\label{lemma 4.3}
     Let $(M,\psi(t))_{t\in[0,T]}$ be a solution of the modified Laplacian coflow $\eqref{modified Laplacian coflow}$. There exists a universal constant $C$ such that
     \begin{align}
         \left(\frac{\partial}{\partial t}-\triangle\right)\mathbf{A}_{N}&\leq-\frac{7}{4}\tilde{\mathbf{A}}_{N+1}+\frac{3}{16}\Psi_{N+1}+C\Psi_{N}\left(t\Phi_{N}+t^{2}\Phi_{N}^{2}+t^{3}\Phi_{N}^{3}+t^{4}\Phi_{N}^{4}\right)\notag\\
        &\quad +C\Phi_{N}^{4}\left(1+t\Phi_{N}+t^{2}\Phi_{N}^{2}+t^{3}\Phi_{N}^{3}\right)\notag.
     \end{align}
\end{lemma}
\begin{proof}

Applying Proposition \ref{prop 3.8} to $S={\rm Rm}(p=4)$, and subsequently applying $\nabla^{k}$ to Lemma \ref{lemma 3.2}, we obtain
    \begin{align}
        \left|\left(\frac{\partial}{\partial t}-\triangle\right)\nabla^{k}{\rm Rm}\right| &\leq C\sum_{i=0}^{k}\frac{(k+2)!}{(i+2)!(k-i)!}|\nabla^{i}{\rm Rm}||\nabla^{k-i}{\rm Rm}|\notag\\
        &\quad+C\sum_{i=0}^{k}\frac{(k+1)!}{(i+1)!(k-i)!}|\nabla^{i}(\mathbf{T}\ast\mathbf{T}\ast\varphi\ast\varphi)||\nabla^{k-i}{\rm Rm}|\notag\\
        &\quad+C\sum_{i=0}^{k}\frac{(k+1)!}{(i+1)!(k-i)!}A|\nabla^{i}\mathbf{T}||\nabla^{k-i}{\rm Rm}|\notag\\
        &\quad+C\sum_{i=0}^{k}\frac{(k+1)!}{(i+1)!(k-i)!}|\nabla^{i}(\mathbf{T}\ast\mathbf{T})||\nabla^{k-i}{\rm Rm}|\notag\\
        &\quad+C|\nabla^{k+2}(\mathbf{T}\ast\mathbf{T}\ast\varphi\ast\varphi)|+CA|\nabla^{k+2}\mathbf{T}|+C|\nabla^{k+2}(\mathbf{T}\ast\mathbf{T})|\notag.
    \end{align}
    Note that $$|\nabla^{k}{\rm Rm}|^{2}=\left(g^{-1}\right)^{k+4}\ast\nabla^{k}{\rm Rm}\ast\nabla^{k}{\rm Rm},$$
    then we can provide the following estimate
    \begin{align}
        \left(\frac{\partial}{\partial t}-\triangle\right)|\nabla^{k}{\rm Rm}|^{2} &=-2|\nabla^{k+1}{\rm Rm}|^{2}+2\left\langle\left(\frac{\partial}{\partial t}-\triangle\right)\nabla^{k}{\rm Rm},\nabla^{k}{\rm Rm}\right\rangle\notag\\
        &\quad +2(k+4)\left({\rm Ric}+\frac{1}{2}\mathbf{T}\ast\mathbf{T}\ast\varphi\ast\varphi+2A\ast\mathbf{T}+\mathbf{T}\ast\mathbf{T}\right)\notag\\
        &\quad\ast\nabla^{k}{\rm Rm}\ast\nabla^{k}{\rm Rm}\notag\\
        &\leq  C\sum_{i=0}^{k}\frac{(k+2)!}{(i+2)!(k-i)!}|\nabla^{i}{\rm Rm}||\nabla^{k-i}{\rm Rm}||\nabla^{k}{\rm Rm}|\notag\\
        &\quad+C\sum_{i=0}^{k}\frac{(k+1)!}{(i+1)!(k-i)!}|\nabla^{i}(\mathbf{T}\ast\mathbf{T}\ast\varphi\ast\varphi)||\nabla^{k-i}{\rm Rm}||\nabla^{k}{\rm Rm}|\notag\\
        &\quad+C\sum_{i=0}^{k}\frac{(k+1)!}{(i+1)!(k-i)!}A|\nabla^{i}\mathbf{T}||\nabla^{k-i}{\rm Rm}||\nabla^{k}{\rm Rm}|\notag\\
        &\quad+C\sum_{i=0}^{k}\frac{(k+1)!}{(i+1)!(k-i)!}|\nabla^{i}(\mathbf{T}\ast\mathbf{T})||\nabla^{k-i}{\rm Rm}||\nabla^{k}{\rm Rm}|\notag\\
        &\quad+C|\nabla^{k+2}(\mathbf{T}\ast\mathbf{T}\ast\varphi\ast\varphi)||\nabla^{k}{\rm Rm}|+CA|\nabla^{k+2}\mathbf{T}||\nabla^{k}{\rm Rm}|\notag\\
        &\quad+C|\nabla^{k+2}(\mathbf{T}\ast\mathbf{T})||\nabla^{k}{\rm Rm}|-2|\nabla^{k+1}{\rm Rm}|^{2}\notag.
    \end{align}
    Thus by the definition of $a_{k}$ in $\eqref{def 1.1}$, it follows that
    \begin{align}\label{4.8}
        \left(\frac{\partial}{\partial t}-\triangle\right)a_{k}^{2}&\leq -2\tilde{a}_{k+1}^{2}+\frac{k\tilde{a}_{k}^{2}}{(k+1)^{2}}\\
        &\quad+I_{1}(k)+I_{2}(k)+I_{3}(k)+I_{4}(k)+I_{5}(k)+I_{6}(k),\notag
    \end{align}
    where
    \begin{align}
        I_{1}(k)&=\frac{Ct^{k}}{(k+1)!^{2}}\sum_{i=0}^{k}\frac{(k+2)!}{(i+2)!(k-i)!}|\nabla^{i}{\rm Rm}||\nabla^{k-i}{\rm Rm}||\nabla^{k}{\rm Rm}|\notag,\\
        I_{2}(k)&=\frac{Ct^{k}}{(k+1)!^{2}}\sum_{i=0}^{k}\frac{(k+1)!}{(i+1)!(k-i)!}|\nabla^{i}(\mathbf{T}\ast\mathbf{T}\ast\varphi\ast\varphi)||\nabla^{k-i}{\rm Rm}||\nabla^{k}{\rm Rm}|\notag,\\
        I_{3}(k)&=\frac{Ct^{k}}{(k+1)!^{2}}\sum_{i=0}^{k}\frac{(k+1)!}{(i+1)!(k-i)!}A|\nabla^{i}\mathbf{T}||\nabla^{k-i}{\rm Rm}||\nabla^{k}{\rm Rm}|\notag,\\
        I_{4}(k)&=\frac{Ct^{k}}{(k+1)!^{2}}\sum_{i=0}^{k}\frac{(k+1)!}{(i+1)!(k-i)!}|\nabla^{i}(\mathbf{T}\ast\mathbf{T})||\nabla^{k-i}{\rm Rm}||\nabla^{k}{\rm Rm}|\notag,\\
        I_{5}(k)&=\frac{Ct^{k}}{(k+1)!^{2}}|\nabla^{k+1}(\nabla\mathbf{T}\ast\mathbf{T}\ast\varphi\ast\varphi+\mathbf{T}\ast\mathbf{T}\ast\nabla\varphi\ast\varphi)||\nabla^{k}{\rm Rm}|\notag,\\
        I_{6}(k)&=\frac{Ct^{k}}{(k+1)!^{2}}\left(A|\nabla^{k+2}\mathbf{T}||\nabla^{k}{\rm Rm}|+|\nabla^{k+2}(\mathbf{T}\ast\mathbf{T})|\right)|\nabla^{k}{\rm Rm}|.\notag
    \end{align}
    To get the estimate for $\mathbf{A}_{N}$, we should sum $\eqref{4.8}$ from $0$ to $N$, for $k=0$, we have
    \begin{align}
        \left(\frac{\partial}{\partial t}-\triangle\right)a_{0}^{2}&\leq -2\tilde{a}_{1}^{2}+C\left(a_{0}^{3}+a_{0}^{2}|\mathbf{T}|^{2}|\varphi|^{2}+Aa_{0}^{2}|\mathbf{T}|+a_{0}^{2}|\mathbf{T}|^{2}+a_{0}\tilde{b}_{1}|\mathbf{T}||\varphi|^{2}\right.\notag\\
        &\quad+a_{0}b_{0}^{2}|\varphi|^{2}+a_{0}b_{0}|\mathbf{T}||\nabla\varphi||\varphi|+a_{0}c_{0}|\mathbf{T}|^{2}|\varphi|+a_{0}|\mathbf{T}|^{2}|\nabla\varphi|^{2}\notag\\
        &\quad\left.+Aa_{0}\tilde{b}_{1}+a_{0}\tilde{b}_{1}|\mathbf{T}|+a_{0}b_{0}^{2}\right)\notag\\
        &\leq -2\tilde{a}_{1}^{2}+C\sum_{i=3}^{7}\Phi_{N}^{\frac{i}{2}}+C\tilde{\mathbf{B}}_{N}^{\frac{1}{2}}(\Phi_{N}+\Phi_{N}^{2})\notag,
    \end{align}
    where the last inequality uses $|\nabla\varphi|\leq|\mathbf{T}||\psi|$ and Cauchy-Schwarz inequality.

    For $k=1$ to $N$, we should calculate the followings:
    \begin{align}
        \sum_{k=1}^{N}I_{1}(k)&\leq C\sum_{k=1}^{N}t^{\frac{1}{2}}\tilde{a}_{k}\sum_{i=0}^{k}t^{\frac{1}{2}}a_{i}\tilde{a}_{k-i}\\
        &\leq Ct\left(\sum_{k=1}^{N}\tilde{a}_{k}^{2}\right)^{\frac{1}{2}}\left(\sum_{k=1}^{N}\sum_{i=0}^{k}a_{i}^{2}\tilde{a}_{k-i}^{2}\right)^{\frac{1}{2}}\notag\\
        &\leq Ct\tilde{\mathbf{A}}_{N}^{\frac{1}{2}}\left(\tilde{\mathbf{A}}_{N}+\tilde{a}_{0}^{2}\right)^{\frac{1}{2}}\mathbf{A}_{N}^{\frac{1}{2}}\notag\\
        &\leq Ct\Psi_{N}\Phi_{N}^{\frac{1}{2}}+Ct^{\frac{1}{2}}\Psi_{N}^{\frac{1}{2}}\Phi_{N}\notag,\\
        \sum_{k=1}^{N}I_{2}(k)&\leq C\sum_{k=1}^{N}a_{k}\sum_{i=1}^{k}t^{\frac{1}{2}}\tilde{a}_{k-i}\sum_{j+l+m+n=i}t^{3}b_{j-1}b_{l-1}c_{m-2}c_{n-2}\\
        &\quad+C\sum_{k=1}^{N}t^{\frac{1}{2}}a_{k}\tilde{a}_{k}|\mathbf{T}|^{2}|\varphi|^{2}\notag\\
        &\leq Ct^{\frac{1}{2}}\mathbf{A}_{N}^{\frac{1}{2}}\left(\tilde{\mathbf{A}}_{N}+\tilde{a}_{0}^{2}\right)^{\frac{1}{2}}P(2,0,2,0)+Ct^{\frac{1}{2}}\tilde{\mathbf{A}}_{N}^{\frac{1}{2}}\mathbf{A}_{N}^{\frac{1}{2}}|\mathbf{T}|^{2}|\varphi|^{2}\notag\\
        &\leq C\sum_{i=1}^{7}t^{\frac{i}{2}}\Psi_{N}^{\frac{1}{2}}\Phi_{N}^{\frac{5}{2}}+C\sum_{i=1}^{6}t^{\frac{i}{2}}\Phi_{N}^{3}\notag,\\
        \sum_{k=1}^{N}I_{3}(k)&\leq C\sum_{k=1}^{N}a_{k}A\sum_{i=0}^{k}t\tilde{a}_{k-i}b_{i-1}\\
        &\leq CtA\mathbf{A}_{N}^{\frac{1}{2}}\left(\tilde{\mathbf{A}}_{N}+\tilde{a}_{0}^{2}\right)^{\frac{1}{2}}\left(\mathbf{B}_{N}+b_{-1}^{2}\right)^{\frac{1}{2}}\notag\\
        &\leq C(t^{\frac{1}{2}}+t)\Psi_{N}^{\frac{1}{2}}\Phi_{N}^{\frac{3}{2}}+C(1+t^{\frac{1}{2}})\Phi_{N}^{2}\notag,\\
        \sum_{k=1}^{N}I_{4}(k)&\leq C\sum_{k=1}^{N}a_{k}\sum_{i=1}^{k}t^{\frac{1}{2}}\tilde{a}_{k-i}\sum_{j+l=i}tb_{j-1}b_{l-1}+C\sum_{k=1}^{N}a_{k}^{2}|\mathbf{T}|^{2}\\
        &\leq Ct^{\frac{1}{2}}\mathbf{A}_{N}^{\frac{1}{2}}\left(\tilde{\mathbf{A}}_{N}+\tilde{a}_{0}^{2}\right)^{\frac{1}{2}}P(2,0,0,0)+C\mathbf{A}_{N}|\mathbf{T}|^{2}\notag\\
        &\leq C\sum_{i=2}^{3}t^{\frac{i}{2}}\Psi_{N}^{\frac{1}{2}}\Phi_{N}^{\frac{3}{2}}+C\sum_{i=0}^{2}t^{\frac{i}{2}}\Phi_{N}^{2}\notag,\\
        \sum_{k=1}^{N}I_{5}(k)&\leq C\sum_{k=1}^{N}a_{k}\sum_{i=0}^{k}t^{\frac{1}{2}}\tilde{b}_{i}\sum_{j+m+n=k+1-i}t^{2}b_{j-1}c_{m-2}c_{n-2}\\
        &\quad+C\sum_{k=1}^{N}t^{\frac{1}{2}}\tilde{a}_{k}\sum_{i+j+m+n=k+1}t^{2}b_{i-1}b_{j-1}c_{m-1}c_{n-2}\notag\\
        &\quad+C\sum_{k=1}^{N}a_{k}\tilde{b}_{k+1}|\mathbf{T}||\varphi|^{2}\notag\\
        &\leq C\mathbf{A}_{N}^{\frac{1}{2}}\left(\tilde{\mathbf{B}}_{N}+\tilde{b}_{0}^{2}\right)^{\frac{1}{2}}P(1,0,2,0)+C\tilde{\mathbf{A}}_{N}^{\frac{1}{2}}P(2,1,1,0)\notag\\
        &\quad+C\mathbf{A}_{N}^{\frac{1}{2}}\tilde{\mathbf{B}}_{N+1}^{\frac{1}{2}}|\mathbf{T}||\varphi|^{2}\notag\\
        &\leq \frac{1}{16}\Psi_{N+1}+C\Phi_{N}^{4}+C\sum_{i=1}^{5}t^{\frac{i}{2}}\Psi_{N}^{\frac{1}{2}}\Phi_{N}^{2}+C\sum_{i=0}^{4}t^{\frac{i}{2}}\Phi_{N}^{\frac{5}{2}}\notag,\\
        \sum_{k=1}^{N}I_{6}(k)&\leq C\sum_{k=1}^{N}t^{\frac{1}{2}}\tilde{a}_{k}\sum_{i=1}^{k+1}b_{i-1}b_{k+1-i}+C\sum_{k=1}^{N}a_{k}\tilde{b}_{k+1}(A+|\mathbf{T}|)\\
        &\leq Ct^{\frac{1}{2}}\tilde{\mathbf{A}}_{N}^{\frac{1}{2}}\mathbf{B}_{N}+C\mathbf{A}_{N}^{\frac{1}{2}}\tilde{\mathbf{B}}_{N+1}^{\frac{1}{2}}(A+|\mathbf{T}|)\notag\\
        &\leq \frac{1}{16}\Psi_{N+1}+C\Phi_{N}^{2}+Ct^{\frac{1}{2}}\Psi_{N}^{\frac{1}{2}}\Phi_{N}\notag.
    \end{align}
    Together with the above, since $\Phi_{N}>1$, we have 
    $$\Phi_{N}^{i}\leq \Phi_{N}^{j},\quad \quad \text{for}\ i\leq j,$$
    by the Cauchy-Schwarz inequality and
    $$\frac{k}{(k+1)^{2}}\leq\frac{1}{4},\quad \tilde{\mathbf{B}}_{N}\leq \tilde{\mathbf{B}}_{N+1}\leq \Psi_{N+1}.$$
    It follows that
    \begin{align}
        \left(\frac{\partial}{\partial t}-\triangle\right)\mathbf{A}_{N}&\leq -2\tilde{\mathbf{A}}_{N+1}+\sum_{k=1}^{N}\frac{k}{(k+1)^{2}}\tilde{a}_{k}^{2}+C\tilde{\mathbf{B}}_{N}^{\frac{1}{2}}\Phi_{N}^{2}+\frac{1}{8}\Psi_{N+1}\notag\\
        &\quad +C\Psi_{N}^{\frac{1}{2}}\left(t^{\frac{1}{2}}\Phi_{N}+\sum_{i=1}^{3}t^{\frac{i}{2}}\Phi_{N}^{\frac{3}{2}}+\sum_{i=1}^{5}t^{\frac{i}{2}}\Phi_{N}^{2}+\sum_{i=1}^{7}t^{\frac{i}{2}}\Phi_{N}^{\frac{5}{2}}\right)\notag\\
        &\quad+Ct\Psi_{N}\Phi_{N}^{\frac{1}{2}}+C\sum_{i=0}^{2}t^{\frac{i}{2}}\Phi_{N}^{2}+C\sum_{i=0}^{4}t^{\frac{i}{2}}\Phi_{N}^{\frac{5}{2}}+C\sum_{i=1}^{6}t^{\frac{i}{2}}\Phi_{N}^{3}+C\Phi_{N}^{4}\notag\\
        &\leq-\frac{7}{4}\tilde{\mathbf{A}}_{N+1}+\frac{3}{16}\Psi_{N+1}+C\Psi_{N}\left(t\Phi_{N}+t^{2}\Phi_{N}^{2}+t^{3}\Phi_{N}^{3}+t^{4}\Phi_{N}^{4}\right)\notag\\
        &\quad +C\Phi_{N}^{4}\left(1+t\Phi_{N}+t^{2}\Phi_{N}^{2}+t^{3}\Phi_{N}^{3}\right)\notag.
    \end{align}
    Thus we obtain the desired result.
\end{proof}

\begin{lemma}\label{lemma 4.4}
     Let $(M,\psi(t))_{t\in[0,T]}$ be a solution of the modified Laplacian coflow $\eqref{modified Laplacian coflow}$. There exists a universal constant $C$ such that
     \begin{align}
         \left(\frac{\partial}{\partial t}-\triangle\right)\mathbf{B}_{N}&\leq -\frac{7}{4}\tilde{\mathbf{B}}_{N+1}+\frac{3}{16}\Psi_{N+1}+C\Psi_{N}\left(t\Phi_{N}+t^{2}\Phi_{N}^{2}+t^{3}\Phi_{N}^{3}+t^{4}\Phi_{N}^{4}\right)\notag\\
        &\quad +C\Phi_{N}^{5}\left(1+t\Phi_{N}+t^{2}\Phi_{N}^{2}+t^{3}\Phi_{N}^{3}\right)\notag.
     \end{align}
\end{lemma}
\begin{proof}
    By Lemma \ref{lemma 3.1} and Proposition \ref{prop 3.8} (with $S=\mathbf{T}$ so $p=2$), one gets
    \begin{align}
        \left|\left(\frac{\partial}{\partial t}-\triangle\right)\nabla^{k+1}\mathbf{T}\right| &\leq C\sum_{i=0}^{k+1}\frac{(k+3)!}{(i+2)!(k+1-i)!}|\nabla^{i}{\rm Rm}||\nabla^{k+1-i}\mathbf{T}|\notag\\
        &\quad+C\sum_{i=0}^{k+1}\frac{(k+2)!}{(i+1)!(k+1-i)!}|\nabla^{i}(\mathbf{T}\ast\mathbf{T}\ast\varphi\ast\varphi)||\nabla^{k+1-i}\mathbf{T}|\notag\\
        &\quad+C\sum_{i=0}^{k+1}\frac{(k+2)!}{(i+1)!(k+1-i)!}A|\nabla^{i}\mathbf{T}||\nabla^{k+1-i}\mathbf{T}|\notag\\
        &\quad+C\sum_{i=0}^{k+1}\frac{(k+2)!}{(i+1)!(k+1-i)!}|\nabla^{i}(\mathbf{T}\ast\mathbf{T})||\nabla^{k+1-i}\mathbf{T}|\notag\\
        &\quad +C|\nabla^{k+1}(A\ast\nabla\mathbf{T}\ast\varphi)|+C|\nabla^{k+1}(\mathbf{T}\ast\nabla\mathbf{T}\ast\varphi)|\notag\\
        &\quad+C|\nabla^{k+1}(\nabla\mathbf{T}\ast\mathbf{T}\ast\varphi\ast\varphi\ast\varphi)|+C|\nabla^{k+1}(\mathbf{T}\ast\nabla\varphi\ast\nabla\varphi)|\notag.
    \end{align}
    Then by the definition of $b_{k}$ in (4.2), similar to Lemma \ref{lemma 4.3}, it follows that
     \begin{align}\label{4.14}
        \left(\frac{\partial}{\partial t}-\triangle\right)b_{k}^{2}&\leq -2\tilde{b}_{k+1}^{2}+\frac{k\tilde{b}_{k}^{2}}{(k+1)^{2}}\\
        &\quad+J_{1}(k)+J_{2}(k)+J_{3}(k)+J_{4}(k)+J_{5}(k)+J_{6}(k),\notag
    \end{align}
 where
    \begin{align}
        J_{1}(k)&=\frac{Ct^{k}}{(k+1)!^{2}}\sum_{i=0}^{k+1}\frac{(k+3)!}{(i+2)!(k+1-i)!}|\nabla^{i}{\rm Rm}||\nabla^{k+1-i}\mathbf{T}||\nabla^{k+1}\mathbf{T}|\notag,\\
        J_{2}(k)&=\frac{Ct^{k}}{(k+1)!^{2}}\sum_{i=0}^{k+1}\frac{(k+2)!}{(i+1)!(k+1-i)!}|\nabla^{i}(\mathbf{T}\ast\mathbf{T}\ast\varphi\ast\varphi)||\nabla^{k+1-i}\mathbf{T}||\nabla^{k+1}\mathbf{T}|\notag,\\
        J_{3}(k)&=\frac{Ct^{k}}{(k+1)!^{2}}\sum_{i=0}^{k+1}\frac{(k+2)!}{(i+1)!(k+1-i)!}A|\nabla^{i}\mathbf{T}||\nabla^{k+1-i}\mathbf{T}||\nabla^{k+1}\mathbf{T}|\notag,\\
        J_{4}(k)&=\frac{Ct^{k}}{(k+1)!^{2}}\sum_{i=0}^{k+1}\frac{(k+2)!}{(i+1)!(k+1-i)!}|\nabla^{i}(\mathbf{T}\ast\mathbf{T})||\nabla^{k+1-i}\mathbf{T}||\nabla^{k+1}\mathbf{T}|\notag,\\
        J_{5}(k)&=\frac{Ct^{k}}{(k+1)!^{2}}\left(|\nabla^{k+1}(A\ast\nabla\mathbf{T}\ast\varphi)|+|\nabla^{k+1}(\mathbf{T}\ast\nabla\mathbf{T}\ast\varphi)|\right)|\nabla^{k+1}\mathbf{T}|\notag,\\
        J_{6}(k)&=\frac{Ct^{k}}{(k+1)!^{2}}|\nabla^{k+1}(\nabla\mathbf{T}\ast\mathbf{T}\ast\varphi\ast\varphi\ast\varphi)||\nabla^{k+1}\mathbf{T}|\notag,\\
        J_{7}(k)&=\frac{Ct^{k}}{(k+1)!^{2}}|\nabla^{k+1}(\mathbf{T}\ast\nabla\varphi\ast\nabla\varphi)||\nabla^{k+1}\mathbf{T}|.\notag
    \end{align}
    Similarly, we should sum $\eqref{4.14}$ from $0$ to $N$, for $k=0$, by Cauchy-Schwarz inequality, one gets
    \begin{align}
        \left(\frac{\partial}{\partial t}-\triangle\right)b_{0}^{2}&\leq -2\tilde{b}_{1}^{2}+C\left(a_{0}b_{0}^{2}+\tilde{a}_{1}b_{0}|\mathbf{T}|+b_{0}^{2}|\mathbf{T}|^{2}|\varphi|^{2}+b_{0}|\mathbf{T}|^{4}|\varphi||\psi|+Ab_{0}^{2}|\mathbf{T}|\right.\notag\\
        &\quad+b_{0}^{2}|\mathbf{T}|^{2}+A\tilde{b}_{1}b_{0}|\varphi|+Ab_{0}^{2}|\mathbf{T}||\psi|+b_{0}^{3}|\varphi|+\tilde{b}_{1}b_{0}|\mathbf{T}||\varphi|+b_{0}^{2}|\mathbf{T}|^{2}|\psi|\notag\\
        &\quad\left.+\tilde{b}_{1}b_{0}|\mathbf{T}||\varphi|^{3}+b_{0}^{3}|\varphi|^{3}+b_{0}^{2}|\mathbf{T}|^{2}|\varphi|^{2}|\psi|+b_{0}^{2}|\mathbf{T}|^{2}|\psi|^{2}+b_{0}c_{0}|\mathbf{T}^{2}|\psi|\right)\notag\\
        &\leq -2\tilde{b}_{1}^{2}+\frac{1}{16}\Psi_{N}+C\Phi_{N}^{5}.\notag
    \end{align}
    where we use $|\nabla\varphi|\leq |\mathbf{T}||\psi|$ and $\Phi_{N}>1$.

    For $k=1$ to $N$, by the Cauchy-Schwarz inequality and
    $$\frac{k+3}{(i+2)(k+1-i)}=\frac{1}{i+2}+\frac{1}{k+1-i}.$$
    It follows that
    \begin{align}
    \sum_{k=1}^{N}J_{1}(k)&\leq C\sum_{k=1}^{N}t^{\frac{1}{2}}\tilde{b}_{k}\sum_{i=0}^{k}\frac{(k+2)(k+3)}{(k+1)(i+2)(k+1-i)}t^{\frac{1}{2}}a_{i}\tilde{b}_{k-i}\\
    &\quad+C\sum_{k=1}^{N}\tilde{a}_{k+1}b_{k}|\mathbf{T}|\notag\\
    &\leq Ct\tilde{\mathbf{B}}_{N}^{\frac{1}{2}}\mathbf{A}_{N}^{\frac{1}{2}}\left(\tilde{\mathbf{B}}_{N}+\tilde{b}_{0}^{2}\right)^{\frac{1}{2}}+C\tilde{\mathbf{A}}_{N+1}^{\frac{1}{2}}\mathbf{B}_{N}^{\frac{1}{2}}|\mathbf{T}|\notag\\
    &\leq Ct\Psi_{N}\Phi_{N}^{\frac{1}{2}}+Ct^{\frac{1}{2}}\Psi_{N}^{\frac{1}{2}}\Phi_{N}+\frac{1}{16}\Psi_{N+1}+C\Phi_{N}^{2}\notag,\\
    \sum_{k=1}^{N}J_{2}(k)&\leq C\sum_{k=1}^{N}t^{\frac{1}{2}}\tilde{b}_{k}\sum_{i=0}^{k+1}b_{k-i}\sum_{j+l+m+n=i}t^{3}b_{j-1}b_{l-1}c_{m-2}c_{n-2}\\
    &\leq C\tilde{\mathbf{B}}_{N}^{\frac{1}{2}}P(3,0,2,0)\notag\\
    &\leq C\sum_{i=1}^{7}t^{\frac{i}{2}}\Psi_{N}^{\frac{1}{2}}\Phi_{N}^{\frac{5}{2}}\notag,\\
    \sum_{k=1}^{N}J_{3}(k)&\leq C\sum_{k=1}^{N}t^{\frac{1}{2}}\tilde{b}_{k}\sum_{i=0}^{k+1}t^{\frac{1}{2}}Ab_{i-1}b_{k-i}\\
    &\leq C\tilde{\mathbf{B}}_{N}^{\frac{1}{2}}AP(2,0,0,0)\notag\\
    &\leq C\sum_{i=1}^{2}t^{\frac{i}{2}}\Psi_{N}^{\frac{1}{2}}\Phi_{N}^{\frac{3}{2}}\notag,\\
    \sum_{k=1}^{N}J_{4}(k)&\leq C\sum_{k=1}^{N}t^{\frac{1}{2}}\tilde{b}_{k}\sum_{i=0}^{k+1}b_{k-i}\sum_{j+l=i}tb_{j-1}b_{l-1}\\
    &\leq C\tilde{\mathbf{B}}_{N}^{\frac{1}{2}}P(3,0,0,0)\notag\\
    &\leq C\sum_{i=1}^{3}t^{\frac{i}{2}}\Psi_{N}^{\frac{1}{2}}\Phi_{N}^{\frac{3}{2}}\notag,\\
    \sum_{k=1}^{N}J_{5}(k)&\leq C\sum_{k=1}^{N}b_{k}\sum_{i=0}^{k}tA\tilde{b}_{i}c_{k-i-1}+C\sum_{k=1}^{N}b_{k}\tilde{b}_{k+1}A|\varphi|\\
    &\quad+ C\sum_{k=1}^{N}b_{k}\sum_{i=0}^{k}t^{\frac{1}{2}}\tilde{b}_{i}\sum_{j+l=k+1-i}tb_{j-1}c_{l-2}+C\sum_{k=1}^{N}b_{k}\tilde{b}_{k+1}|\mathbf{T}||\varphi|\notag\\
    &\leq CtA\mathbf{B}_{N}^{\frac{1}{2}}\left(\tilde{\mathbf{B}}_{N}+\tilde{b}_{0}^{2}\right)^{\frac{1}{2}}\left(\mathbf{C}_{N}+c_{-1}^{2}\right)^{\frac{1}{2}}+C\mathbf{B}_{N}^{\frac{1}{2}}\tilde{\mathbf{B}}_{N+1}^{\frac{1}{2}}A|\varphi|\notag\\
    &\quad+C\mathbf{B}_{N}^{\frac{1}{2}}\left(\tilde{\mathbf{B}}_{N}+\tilde{b}_{0}^{2}\right)^{\frac{1}{2}}P(1,0,1,0)+C\mathbf{B}_{N}^{\frac{1}{2}}\tilde{\mathbf{B}}_{N+1}^{\frac{1}{2}}|\mathbf{T}||\varphi|\notag\\
    &\leq C\sum_{i=1}^{3}t^{\frac{i}{2}}\Psi_{N}^{\frac{1}{2}}\Phi_{N}^{\frac{3}{2}}+ C\sum_{i=0}^{2}t^{\frac{i}{2}}\Phi_{N}^{2}+\frac{1}{32}\Psi_{N+1}+C\Phi_{N}^{3}\notag,\\
    \sum_{k=1}^{N}J_{6}(k)&\leq C\sum_{k=1}^{N}b_{k}\sum_{i=0}^{k}t^{\frac{1}{2}}\tilde{b}_{i}\sum_{j+l+m+n=k+1-i}t^{3}b_{j-1}c_{l-2}c_{m-2}c_{n-2}\\
    &\quad+C\sum_{k=1}^{N}b_{k}\tilde{b}_{k+1}|\mathbf{T}||\varphi|^{3}\notag\\
    &\leq C\mathbf{B}_{N}^{\frac{1}{2}}\left(\tilde{\mathbf{B}}_{N}+\tilde{b}_{0}^{2}\right)^{\frac{1}{2}}P(1,0,3,0)+C\mathbf{B}_{N}^{\frac{1}{2}}\tilde{\mathbf{B}}_{N+1}^{\frac{1}{2}}|\mathbf{T}||\varphi|^{3}\notag\\
    &\leq C\sum_{i=1}^{7}t^{\frac{i}{2}}\Psi_{N}^{\frac{1}{2}}\Phi_{N}^{\frac{5}{2}}+ C\sum_{i=0}^{6}t^{\frac{i}{2}}\Phi_{N}^{3}+\frac{1}{32}\Psi_{N+1}+C\Phi_{N}^{5}\notag,\\
    \sum_{k=1}^{N}J_{7}(k)&\leq C\sum_{k=1}^{N}t^{\frac{1}{2}}\tilde{b}_{k}\sum_{i+j+l=k+1}tb_{i-1}c_{j-1}c_{l-1}\\
    &\leq C\tilde{\mathbf{B}}_{N}^{\frac{1}{2}}P(1,2,0,0)\notag\\
    &\leq C\sum_{i=1}^{3}t^{\frac{i}{2}}\Psi_{N}^{\frac{1}{2}}\Phi_{N}^{\frac{3}{2}}\notag.
    \end{align}
Submitting the above to $\eqref{4.14}$, {similarly} to Lemma \ref{lemma 4.3}, we get
    \begin{align}
        \left(\frac{\partial}{\partial t}-\triangle\right)\mathbf{B}_{N}
        &\leq-\frac{7}{4}\tilde{\mathbf{B}}_{N+1}+\frac{3}{16}\Psi_{N+1}+C\Psi_{N}\left(t\Phi_{N}+t^{2}\Phi_{N}^{2}+t^{3}\Phi_{N}^{3}+t^{4}\Phi_{N}^{4}\right)\notag\\
        &\quad +C\Phi_{N}^{5}\left(1+t\Phi_{N}+t^{2}\Phi_{N}^{2}+t^{3}\Phi_{N}^{3}\right)\notag.
    \end{align}
    Thus we prove Lemma \ref{lemma 4.4}.
\end{proof}

\begin{lemma}\label{lemma 4.5}
     Let $(M,\psi(t))_{t\in[0,T]}$ be a solution of the modified Laplacian coflow $\eqref{modified Laplacian coflow}$. There exists a universal constant $C$ such that
     \begin{align}
         \left(\frac{\partial}{\partial t}-\triangle\right)\mathbf{C}_{N}&\leq -\frac{7}{4}\tilde{\mathbf{C}}_{N+1}+C\Psi_{N}\left(t\Phi_{N}+t^{2}\Phi_{N}^{2}+t^{3}\Phi_{N}^{3}+t^{4}\Phi_{N}^{4}+t^{5}\Phi_{N}^{5}\right)\notag\\
        &\quad +\frac{3}{16}\Psi_{N+1}+C\Phi_{N}^{5}\left(1+t\Phi_{N}+t^{2}\Phi_{N}^{2}+t^{3}\Phi_{N}^{3}+t^{4}\Phi_{N}^{4}\right)\notag.
     \end{align}
\end{lemma}
\begin{proof}
     By Lemma \ref{lemma 3.3} and Proposition \ref{prop 3.8} (with $S=\nabla\varphi$ so $p=4$), it follows that
     \begin{align}
        \left|\left(\frac{\partial}{\partial t}-\triangle\right)\nabla^{k+2}\varphi\right| &\leq C\sum_{i=0}^{k+1}\frac{(k+3)!}{(i+2)!(k+1-i)!}|\nabla^{i}{\rm Rm}||\nabla^{k+2-i}\varphi|\notag\\
        &\quad+C\sum_{i=0}^{k+1}\frac{(k+2)!}{(i+1)!(k+1-i)!}|\nabla^{i}(\mathbf{T}\ast\mathbf{T}\ast\varphi\ast\varphi)||\nabla^{k+2-i}\varphi|\notag\\
        &\quad+C\sum_{i=0}^{k+1}\frac{(k+2)!}{(i+1)!(k+1-i)!}A|\nabla^{i}\mathbf{T}||\nabla^{k+2-i}\varphi|\notag\\
        &\quad+C\sum_{i=0}^{k+1}\frac{(k+2)!}{(i+1)!(k+1-i)!}|\nabla^{i}(\mathbf{T}\ast\mathbf{T})||\nabla^{k+2-i}\varphi|\notag\\
        &\quad +C|\nabla^{k+1}({\rm Rm}\ast\mathbf{T}\ast\psi)|+C|\nabla^{k+1}(\mathbf{T}\ast\nabla\varphi\ast\nabla\varphi\ast\psi)|\notag\\
        &\quad+C|\nabla^{k+1}(A\ast\mathbf{T}\ast\mathbf{T}\ast\psi)|+C|\nabla^{k+1}(\nabla\mathbf{T}\ast Q(A,\mathbf{T},\varphi,\psi))|\notag\\
&\quad+C|\nabla^{k+1}\left[\mathbf{T}\ast\mathbf{T}\ast\mathbf{T}\ast(\psi+\varphi\ast\varphi\ast\psi)\right]|\notag,
    \end{align}
    where $$Q:=Q(A,\mathbf{T},\varphi,\psi)=\mathbf{T}\ast\varphi+A\ast\varphi\ast\psi+\mathbf{T}\ast\varphi\ast\psi+\mathbf{T}\ast\varphi\ast\varphi\ast\varphi\ast\psi.$$
    Then by the definition of $c_{k}$ in (4.2), similar to Lemma \ref{lemma 4.3}, it follows that
     \begin{align}\label{4.25}
        \left(\frac{\partial}{\partial t}-\triangle\right)c_{k}^{2}&\leq -2\tilde{c}_{k+1}^{2}+\frac{k\tilde{c}_{k}^{2}}{(k+1)^{2}}+K_{1}(k)+K_{2}(k)+K_{3}(k)\\
        &\quad+K_{4}(k)+K_{5}(k)+K_{6}(k)+K_{7}(k)+K_{8}(k),\notag
    \end{align}
     where
    \begin{align}
        K_{1}(k)&=\frac{Ct^{k}}{(k+1)!^{2}}\sum_{i=0}^{k+1}\frac{(k+3)!}{(i+2)!(k+1-i)!}|\nabla^{i}{\rm Rm}||\nabla^{k+2-i}\varphi||\nabla^{k+2}\varphi|\notag,\\
        K_{2}(k)&=\frac{Ct^{k}}{(k+1)!^{2}}\sum_{i=0}^{k+1}\frac{(k+2)!}{(i+1)!(k+1-i)!}|\nabla^{i}(\mathbf{T}\ast\mathbf{T}\ast\varphi\ast\varphi)||\nabla^{k+2-i}\varphi||\nabla^{k+2}\varphi|\notag,\\
        K_{3}(k)&=\frac{Ct^{k}}{(k+1)!^{2}}\sum_{i=0}^{k+1}\frac{(k+2)!}{(i+1)!(k+1-i)!}A|\nabla^{i}\mathbf{T}||\nabla^{k+2-i}\varphi||\nabla^{k+2}\varphi|\notag,\\
        K_{4}(k)&=\frac{Ct^{k}}{(k+1)!^{2}}\sum_{i=0}^{k+1}\frac{(k+2)!}{(i+1)!(k+1-i)!}|\nabla^{i}(\mathbf{T}\ast\mathbf{T})||\nabla^{k+2-i}\varphi||\nabla^{k+2}\varphi|\notag,\\
        K_{5}(k)&=\frac{Ct^{k}}{(k+1)!^{2}}\left(|\nabla^{k+1}({\rm Rm}\ast\mathbf{T}\ast\psi)|+|\nabla^{k+1}(\mathbf{T}\ast\nabla\varphi\ast\nabla\varphi\ast\psi)|\right)|\nabla^{k+2}\varphi|\notag,\\
        K_{6}(k)&=\frac{Ct^{k}}{(k+1)!^{2}}\left(|\nabla^{k+1}(A\ast\mathbf{T}\ast\mathbf{T}\ast\psi)|+|\nabla^{k+1}(\mathbf{T}\ast\mathbf{T}\ast\mathbf{T}\ast\psi)|\right)|\nabla^{k+2}\varphi|\notag,\\
        K_{7}(k)&=\frac{Ct^{k}}{(k+1)!^{2}}|\nabla^{k+1}\left(\nabla\mathbf{T}\ast Q(A,\mathbf{T},\varphi,\psi)\right)||\nabla^{k+2}\varphi|\notag,\\
        K_{8}(k)&=\frac{Ct^{k}}{(k+1)!^{2}}|\nabla^{k+1}\left(\mathbf{T}\ast\mathbf{T}\ast\mathbf{T}\ast\varphi\ast\varphi\ast\psi\right)||\nabla^{k+2}\varphi|\notag.
    \end{align}
    Similarly, we should sum $\eqref{4.25}$ from $0$ to $N$, for $k=0$, by the Cauchy-Schwarz inequality, one gets
    \begin{align}
        \left(\frac{\partial}{\partial t}-\triangle\right)c_{0}^{2}
        &\leq -2\tilde{c}_{1}^{2}+\frac{1}{16}\Psi_{N}+C\Phi_{N}^{5},
    \end{align}
    where we use $|\nabla\varphi|\leq |\mathbf{T}||\psi|$ and $\Phi_{N}>1$.

    For $k=1$ to $N$, we should calculate the followings
    \begin{align}
        \sum_{k=1}^{N}K_{1}(k)&\leq C\sum_{k=1}^{N}t^{\frac{1}{2}}\tilde{c}_{k}\sum_{i=0}^{k}t^{\frac{1}{2}}a_{i}\tilde{c}_{k-i}+C\sum_{k=1}^{N}\tilde{a}_{k+1}c_{k}|\nabla\varphi|\\
        &\leq Ct\tilde{\mathbf{C}}_{N}^{\frac{1}{2}}\mathbf{A}_{N}^{\frac{1}{2}}\left(\tilde{\mathbf{C}}_{N}+\tilde{c}_{0}^{2}\right)^{\frac{1}{2}}+C\tilde{\mathbf{A}}_{N+1}^{\frac{1}{2}}\mathbf{C}_{N}^{\frac{1}{2}}|\mathbf{T}||\psi|\notag\\
        &\leq Ct\Psi_{N}\Phi_{N}^{\frac{1}{2}}+Ct^{\frac{1}{2}}\Psi_{N}^{\frac{1}{2}}\Phi_{N}+\frac{1}{16}\Psi_{N+1}+C\Phi_{N}^{3}\notag,\\
        \sum_{k=1}^{N}K_{2}(k)&\leq C\sum_{k=1}^{N}t^{\frac{1}{2}}\tilde{c}_{k}\sum_{i=0}^{k+1}c_{k-i}\sum_{j+l+m+n=i}t^{3}b_{j-1}b_{l-1}c_{m-2}c_{n-2}\\
        &\leq C\tilde{\mathbf{C}}_{N}^{\frac{1}{2}}P(2,1,2,0)\notag\\
        &\leq C\sum_{i=1}^{7}t^{\frac{i}{2}}\Psi_{N}^{\frac{1}{2}}\Phi_{N}^{\frac{5}{2}}\notag,\\
        \sum_{k=1}^{N}K_{3}(k)&\leq C\sum_{k=1}^{N}t^{\frac{1}{2}}\tilde{c}_{k}\sum_{i=0}^{k+1}t^{\frac{1}{2}}Ab_{i-1}c_{k-i}\\
        &\leq CA\tilde{\mathbf{C}}_{N}^{\frac{1}{2}}P(1,1,0,0)\notag\\
        &\leq C\sum_{i=1}^{2}t^{\frac{i}{2}}\Psi_{N}^{\frac{1}{2}}\Phi_{N}^{\frac{3}{2}}\notag,\\
        \sum_{k=1}^{N}K_{4}(k)&\leq C\sum_{k=1}^{N}t^{\frac{1}{2}}\tilde{c}_{k}\sum_{i=0}^{k+1}c_{k-i}\sum_{j+l=i}tb_{j-1}b_{l-1}\\
        &\leq C\tilde{\mathbf{C}}_{N}^{\frac{1}{2}}P(2,1,0,0)\notag\\
        &\leq C\sum_{i=1}^{3}t^{\frac{i}{2}}\Psi_{N}^{\frac{1}{2}}\Phi_{N}^{\frac{3}{2}}\notag,\\
        \sum_{k=1}^{N}K_{5}(k)&\leq C\sum_{k=1}^{N}c_{k}\sum_{i=0}^{k}t^{\frac{1}{2}}\tilde{a}_{i}\sum_{j+l=k+1-i}tb_{j-1}d_{l-2}+C\sum_{k=1}^{N}\tilde{a}_{k+1}c_{k}|\mathbf{T}||\psi|\\
        &\quad+C\sum_{k=1}^{N}c_{k}\sum_{i+j+l+m=k+1}t^{2}b_{i-1}c_{j-1}c_{l-1}d_{m-2}\notag\\
        &\leq C\left(\tilde{\mathbf{A}}_{N}+\tilde{a}_{0}^{2}\right)^{\frac{1}{2}}\mathbf{C}_{N}^{\frac{1}{2}}P(1,0,0,1)+C\tilde{\mathbf{A}}_{N+1}^{\frac{1}{2}}\mathbf{C}_{N}^{\frac{1}{2}}|\mathbf{T}||\psi|\notag\\
        &\quad+Ct^{-\frac{1}{2}}\mathbf{C}_{N}^{\frac{1}{2}}P(1,2,0,1)\notag\\
        &\leq \frac{1}{16}\Psi_{N+1}+C\Phi_{N}^{3}+C\sum_{i=1}^{3}t^{\frac{i}{2}}\Psi_{N}^{\frac{1}{2}}\Phi_{N}^{\frac{3}{2}}+C\sum_{i=0}^{4}t^{\frac{i}{2}}\left(\Phi_{N}^{2}+\Phi_{N}^{\frac{5}{2}}\right)\notag,\\
        \sum_{k=1}^{N}K_{6}(k)&\leq C\sum_{k=1}^{N}t^{\frac{1}{2}}\tilde{c}_{k}\sum_{i+j+l=k+1}t^{\frac{3}{2}}Ab_{i-1}b_{j-1}d_{l-2}\\
        &\quad+C\sum_{k=1}^{N}t^{\frac{1}{2}}\tilde{c}_{k}\sum_{i+j+l+m=k+1}t^{2}b_{i-1}b_{j-1}b_{l-1}d_{m-2}\notag\\
        &\leq CA\tilde{\mathbf{C}}_{N}^{\frac{1}{2}}P(2,0,0,1)+C\tilde{\mathbf{C}}_{N}^{\frac{1}{2}}P(3,0,0,1)\notag\\
        &\leq C\sum_{i=1}^{5}t^{\frac{i}{2}}\Psi_{N}^{\frac{1}{2}}\Phi_{N}^{2}\notag,\\
        \sum_{k=1}^{N}K_{8}(k)&\leq C\sum_{k=1}^{N}t^{\frac{1}{2}}\tilde{c}_{k}\sum_{i+j+l+m+n+s=k+1}t^{4}b_{i-1}b_{j-1}b_{l-1}c_{m-2}c_{n-2}d_{s-2}\\
        &\leq C\tilde{\mathbf{C}}_{N}^{\frac{1}{2}}P(3,0,2,1)\notag\\
        &\leq C\sum_{i=1}^{9}t^{\frac{i}{2}}\Psi_{N}^{\frac{1}{2}}\Phi_{N}^{3}\notag.
    \end{align}
For $K_{7}(k)$,  we only calculate the following as {an example} (others are similar):
\begin{align}
    &\quad\sum_{k=1}^{N}\frac{Ct^{k}}{(k+1)!^{2}}|\nabla^{k+1}\left(\nabla\mathbf{T}\ast \mathbf{T}\ast\varphi\right)||\nabla^{k+2}\varphi|\\
    &\leq \sum_{k=1}^{N}\frac{Ct^{k}}{(k+1)!^{2}}|\nabla^{k+2}\varphi|\sum_{i=0}^{k+1}\frac{(k+1)!}{i!(k+1-i)!}|\nabla^{i+1}\mathbf{T}||\nabla^{k+1-i}(\mathbf{T}\ast\varphi)|\notag\\
    &\leq C\sum_{k=1}^{N}c_{k}\sum_{i=0}^{k}t^{\frac{1}{2}}\tilde{b}_{i}\sum_{j+l=k+1-i}tb_{j-1}c_{l-2}+C\sum_{k=1}^{N}\tilde{b}_{k+1}c_{k}|\mathbf{T}||\varphi|\notag\\
    &\leq\frac{1}{64}\Psi_{N+1}+C\Phi_{N}^{3}+C\left(\tilde{\mathbf{B}}_{N}+\tilde{b}_{0}\right)^{\frac{1}{2}}\mathbf{C}_{N}^{\frac{1}{2}}P(1,0,1,0)\notag\\
    &\leq \frac{1}{64}\Psi_{N+1}+C\Phi_{N}^{3}+C\sum_{i=1}^{3}t^{\frac{i}{2}}\Psi_{N}^{\frac{1}{2}}\Phi_{N}^{\frac{3}{2}}+C\sum_{i=0}^{2}t^{\frac{i}{2}}\Phi_{N}^{2}\notag.
\end{align}
Thus, we get
\begin{align}
    \sum_{k=1}^{N}K_{7}(k)&\leq \frac{1}{16}\Psi_{N+1}+C\Phi_{N}^{5}+C\sum_{i=1}^{3}t^{\frac{i}{2}}\Psi_{N}^{\frac{1}{2}}\Phi_{N}^{\frac{3}{2}}+C\sum_{i=0}^{2}t^{\frac{i}{2}}\Phi_{N}^{2}\notag\\
    &\quad+C\sum_{i=1}^{5}t^{\frac{i}{2}}\Psi_{N}^{\frac{1}{2}}\Phi_{N}^{2}+C\sum_{i=0}^{4}t^{\frac{i}{2}}\Phi_{N}^{\frac{5}{2}}+C\sum_{i=1}^{9}t^{\frac{i}{2}}\Psi_{N}^{\frac{1}{2}}\Phi_{N}^{3}+C\sum_{i=0}^{8}t^{\frac{i}{2}}\Phi_{N}^{\frac{7}{2}}.\notag
\end{align}
{Substituting} the above to $\eqref{4.25}$, {similarly} to Lemma \ref{lemma 4.3}, we prove Lemma \ref{lemma 4.5}.
\end{proof}

\begin{lemma}\label{lemma 4.6}
     Let $(M,\psi(t))_{t\in[0,T]}$ be a solution of the modified Laplacian coflow $\eqref{modified Laplacian coflow}$. There exists a universal constant $C$ such that
    \begin{align}
         \left(\frac{\partial}{\partial t}-\triangle\right)\mathbf{D}_{N}&\leq -\frac{7}{4}\tilde{\mathbf{D}}_{N+1}+C\Psi_{N}\left(t\Phi_{N}+t^{2}\Phi_{N}^{2}+t^{3}\Phi_{N}^{3}+t^{4}\Phi_{N}^{4}++t^{5}\Phi_{N}^{5}\right)\notag\\
        &\quad +\frac{3}{16}\Psi_{N+1}+C\Phi_{N}^{5}\left(1+t\Phi_{N}+t^{2}\Phi_{N}^{2}+t^{3}\Phi_{N}^{3}+t^{4}\Phi_{N}^{4}\right)\notag.
     \end{align}
\end{lemma}
\begin{proof}
Both the evolution equations of $\nabla\varphi$ and $\nabla\psi$ are similar, so we can obtain Lemma \ref{lemma 4.6} by the same calculation.
\end{proof}

\begin{lemma}
    Let $(M,\psi(t))_{t\in[0,T]}$ be a solution of the modified Laplacian coflow $\eqref{modified Laplacian coflow}$. There exists a universal constant $C$ such that
    \begin{align}
        \left(\frac{\partial}{\partial t}-\triangle\right)\Phi_{N}&\leq -\Psi_{N}+C\Psi_{N}\left[(1+t\Phi_{N})^{5}-1\right]+C\Phi_{N}^{5}(1+t\Phi_{N})^{4}.
    \end{align}
\end{lemma}
\begin{proof}
    From Lemma \ref{lemma 4.3} to Lemma \ref{lemma 4.6}, combining with Lemma \ref{lemma 3.5}, using 
    $$\Phi_{N}^{i}\leq \Phi_{N}^{j},\quad \quad \text{for}\ i\leq j$$
     and
    $$-\Psi_{N+1}\leq-\Psi_{N},$$
    we get the desired result.
\end{proof}

Next, we define
$$\Omega(x,t):=|\nabla^{k}{\rm Rm}|^{2}(x,t)+|\nabla^{k+1}\mathbf{T}|^{2}(x,t)+|\nabla^{k+2}\varphi|^{2}(x,t)+|\nabla^{k+2}\psi|^{2}(x,t),$$
and note that
 $$\Lambda(x,t)=\left(|{\rm Rm}|^{2}(x,t)+|\nabla\mathbf{T}|^{2}(x,t)+|\mathbf{T}|^{4}(x,t)\right)^{\frac{1}{2}}.$$
Thus we get the following improvement of Shi-type estimate:
\begin{theorem}\label{theorem 4.8}
    Let $(M,\psi(t))_{t\in[0,T]}$ be a solution of the modified Laplacian coflow $\eqref{modified Laplacian coflow}$. There exists a positive constant $t_{0}=t_{0}(M_{0},A)$ and $C_{\ast}=C_{\ast}(T,M_{0},A)$, where $M_{0}=\sup_{M}\Lambda(x,0)$, such that
    \begin{align}
        \sum_{k=0}^{N}\frac{t^{k}}{(k+1)!^{2}}\Omega(x,t)\leq C_{\ast}
    \end{align}
    on $M\times\left[0,\min\{T,t_{0},1\}\right]$ for all $N\in \mathbb{N}$.
\end{theorem}
\begin{proof}
    Let $T_{N}$ be the time
    \begin{align}
        T_{N}:=\sup\left\{a\in[0,T]\left|t\Phi_{N}\leq (1+C^{-1})^{\frac{1}{5}}-1, \text{for any}\ (x,t)\in M\times[0,a]\right.\right\}.\notag
    \end{align}
    Then on $M\times[0,T_{N}]$, one gets
    \begin{align}\label{4.39}
        \left(\frac{\partial}{\partial t}-\triangle\right)\Phi_{N}&\leq C\Phi_{N}^{5}(1+C^{-1})^{\frac{4}{5}}\leq (C+1)\Phi_{N}^{5},
    \end{align}
    and note that
    \begin{align}
        \Phi_{N}(x,0)&=|{\rm Rm}|^{2}(x,0)+|\nabla\mathbf{T}|^{2}(x,0)+|\nabla^{2}\varphi|^{2}(x,0)+|\nabla^{2}\psi|^{2}(x,0)\notag\\
        &\quad+|\mathbf{T}|^{2}(x,0)+A^{2}+|\varphi|^{2}(x,0)+|\psi|^{2}(x,0)\notag\\
        &\leq |{\rm Rm}|^{2}(x,0)+|\nabla\mathbf{T}|^{2}(x,0)+2|\nabla\mathbf{T}|^{2}|\psi|^{2}(x,0)+32|\mathbf{T}|^{4}|\varphi|^{2}(x,0)\notag\\
        &\quad+32|\nabla\mathbf{T}|^{2}|\varphi|^{2}(x,0)+32|\mathbf{T}|^{4}|\psi|^{2}(x,0)+|\mathbf{T}|^{2}(x,0)+A^{2}\notag\\
        &\quad+|\varphi|^{2}(x,0)+|\psi|^{2}(x,0)\notag\\
        &\leq 5376\left(\sup_{M}\Lambda(\cdot,0)+A+1\right)^{2}\notag\\
        &=5376\left(M_{0}+A+1\right)^{2}.\notag
    \end{align}
    Since $M$ is compact, we can apply the maximum principle of Theorem 4.4 in \cite{Chow-Knopf 2004} to $\eqref{4.39}$, thus we need to solve the following ODE:
    \begin{align}
        \frac{d}{d t}\Phi_{N}=(C+1)\Phi^{5}_{N}, \quad \Phi_{N}(x,0)\leq 5376\left(M_{0}+A+1\right)^{2}.
    \end{align}
    By Calculating, we get
    \begin{align}\label{4.40}
    \Phi_{N}(x,t)\leq \frac{5376\left(M_{0}+A+1\right)^{2}}{\sqrt[4]{1-4(C+1)t\cdot\left[5376\left(M_{0}+A+1\right)^{2}\right]^{4}}}
    \end{align}
    on  $M\times[0,T_{N}]$. Thus there {exists} a time $t_{0}:=t_{0}(M_{0},A)>0$ satisfies  the definition of $T_{N}$. We denote
    \begin{align}
        T_{0}=T_{0}(T,M_{0},A):=\min\left\{T,t_{0},1\right\}>0.
    \end{align}
    Obviously, $T_{N}\geq T_{0}$ for all $N\in\mathbb{N}$. It follows that $\eqref{4.40}$ holds on $M\times[0,T_{0}]$ for all $N\in\mathbb{N}$. Therefore, we obtain that $\Phi_{N}$ can be bounded by a constant $C_{\ast}=C_{\ast}(T,M_{0},A)$ on $M\times[0,T_{0}]$ for all $N\in\mathbb{N}$. 
\end{proof}

\section{Real analyticity of modified Laplacian coflow}

In this section, we study the real analyticity of the modified Laplacian coflow $\eqref{modified Laplacian coflow}$.
Let ${\rm inj}_{g}(p)$ denote the injectivity radius of a Riemannian metric $g$ at $p\in M$. Choosing coordinates $\{x^{i}\}_{i=1}^{7}$ centered at $p$, we denote
$$\partial^{k}=\sum_{k_{1}+\cdots+k_{7}=k}\frac{\partial^{k_{1}+\cdots+k_{7}}}{(\partial x^{1})^{k_{1}}\cdots(\partial x^{7})^{k_{7}}}.$$
We obtain the following derivatives of $G_{2}$-structure and associate metric in normal coordinates.
\begin{lemma}\label{lemma 5.1}
   Let $(M,\psi,g_{\psi})$ be a $7$-dimensional Riemannian manifold with a $G_{2}$-structure satisfying the following estimate:
   \begin{align}\label{5.1}
       |\nabla^{k}{\rm Rm}|(x)+|\nabla^{k+2}\psi|(x)\leq C_{1}k!r^{-k-2}
   \end{align}
   in a geodesic ball $B(p,r)$ for all $k\in\mathbb{N}$, where $p\in M$, $r>0$ and $C_{1}$ is a positive constant.

   There exist positive constants $C_{2},C_{3},C_{4},r_{1}=r_{1}(r),r_{2}=r_{2}(r)$ such that if we set $\rho=\min\left\{\frac{C_{2}}{\sqrt{C_{1}}}r,{\rm inj}_{g}(p)\right\}$, then for all $x\in B(p,\rho)$ and $k\in\mathbb{N}$, we have in normal coordinates centered at $p$:
   \begin{align}
       \frac{1}{2}\delta_{ij}\leq g_{ij}(x)\leq2 \delta_{ij},\quad&|\partial^{k}g_{ij}|(x)\leq C_{3}k!r_{1}^{-k},\\
       |\partial^{k}\Gamma_{ij}^{l}|(x)\leq C_{3}k!r_{1}^{-k-1},\quad &|\partial^{k}\psi_{ijmn}|(x)\leq C_{4}k!r_{2}^{-k}.
   \end{align}
\end{lemma}
\begin{proof}
    Since $|\nabla^{k}{\rm Rm}|(x)\leq C_{1}k!r^{-k-2}$ in $B(p,r)$ from $\eqref{5.1}$. Thus the proof of Lemma 13.31 in \cite{Ricci flow TA 2} provides the derivative estimates for $g_{ij}$ and $\Gamma_{ij}^{l}$. With the same discussion in Lemma 3.1 of \cite{L-W 2019 2}, we get the derivative estimate for $\psi_{ijmn}$.
\end{proof}
\subsection{Global real analyticity}\label{subsection5.1}
In this section, we provide the real analyticity of the modified Laplacian coflow $\eqref{modified Laplacian coflow}$ in the special case where $M$ is compact.
\begin{theorem}\label{theorem 5.2}
    Let $(M,\psi(t),g_{\psi}(t))_{t\in[0,T]}$ be a smooth solution of the modified Laplacian coflow $\eqref{modified Laplacian coflow}$ with coclosed $G_{2}$-structures on compact $7$-dimensional manifold $M$. For each time $t\in(0,T]$, $(M,\psi(t),g_{\psi}(t))$ is real analytic.
\end{theorem}
\begin{proof}
    From Theorem \ref{theorem 4.8}, one gets
    \begin{align}
        t^{\frac{k}{2}}\left(|\nabla^{k}{\rm Rm}|(x,t)+|\nabla^{k+2}\psi|(x,t)\right)\leq C_{\ast}(k+1)!
    \end{align}
    on $M\times[0,T_{0}]$. Note that $k+1\leq 2^{k}$, for any fixed $t\in (0,T_{0}]$, we denote $r=\sqrt{t}/2$ and $C=C_{\ast}T_{0}/4$ independent on $k$ yields 
    $$|\nabla^{k}{\rm Rm}|(x,t)+|\nabla^{k+2}\psi|(x,t)\leq Ck!r^{-k-2}.$$
    Therefore, Lemma \ref{lemma 5.1} gives the derivative estimates for $g_{ij}$, $\Gamma_{ij}^{l}$ and $\psi_{ijmn}$ in normal coordinates. With the same discussion in Section 2.1(\cite{Ricci flow TA 2}, p.212) and Lemma 13.20 in \cite{Ricci flow TA 2}. It follows that $(M,\psi(t),g_{\psi}(t))$ is real analytic on $(0, T_{0}]$. Since $T_{0}$ is a positive constant, then by finite iteration, we can cover the entire time interval $t\in (0, T]$.
\end{proof}

\subsection{Local real analyticity}\label{subsection5.2}
In this section, we study the local revision real analyticity of the modified Laplacian coflow $\eqref{modified Laplacian coflow}$ via a cut-off function. Firstly, we give the construction of the cut-off function $\eta$.
\begin{lemma}\label{lemma 5.3}
    Let $\psi(t)_{t\in[0,T]}$ be a smooth solution of the modified Laplacian coflow $\eqref{modified Laplacian coflow}$ on an open set $U\subset M$. Let $p\in U$ and $r>0$ such that $\overline{B}_{g_{0}}(p,2r)\subset U$ is compact, where $g_{0}$ is the metric induced by initial coclosed $G_{2}$-structure $\psi$. Suppose that $A^{2}\leq \mathbf{M}$ and
    \begin{align}\label{5.5}
        \left(|{\rm Rm}|^{2}(x,t)+|\nabla\mathbf{T}|^{2}(x,t)+|\mathbf{T}|^{4}(x,t)\right)^{\frac{1}{2}}=\Lambda(x,t)\leq \mathbf{M}
    \end{align}
    for all $(x,t)\in B_{g_{0}}(p,2r)\times[0,T_{\ast}]$ with $0<T_{\ast}\leq\min \{T,1\}$.

    There exists a cut-off function $\eta:U\rightarrow [0,1]$ that is compactly supported in $B_{g_{0}}(p,r)$, satisfies $\eta=1$ in $B_{g_{0}}(p,r/2)$, and its derivatives satisfy
    \begin{align}\label{5.6}
        |\nabla\eta(x)|_{g(t)}^{2}-\eta(x)\triangle_{g(t)}\eta(x)\leq C_{0}\eta(x)
    \end{align}
    on $U\times[0,T_{\ast}]$ for some constant $C_{0}=C_{0}(r,T_{\ast},\mathbf{M})$.
\end{lemma}
\begin{proof}
    Recall that the modified Laplacian coflow $\eqref{modified Laplacian coflow}$ can be written as the Ricci-like flow $\eqref{Ricci-like coflow}$ with velocity $2h(t)$, where $h(t)$ is denoted by $\eqref{geometric flow}$. From $\eqref{5.5}$, it follows that
    \begin{align}
        |h(t)|\leq C\left(|{\rm Rm}|+|\mathbf{T}|^{2}+A^{2}\right)\leq C\mathbf{M},
    \end{align}
    where $C$ is a constant {depends} only on the dimension $7$ and $|\varphi|^{2}=42$. The above implies that
    $$e^{-C(T_{\ast},\mathbf{M})}g_{0}\leq g(t)\leq e^{C(T_{\ast},\mathbf{M})}g_{0}.$$
    Following the proof of Part 1 of Lemma 14.3 in \cite{Ricci flow TA 2}, we get 
    $$|\nabla\eta(x)|_{g(t)}^{2}\leq C_{0}\eta(x)$$
    for all $(x,t)\in B_{g_{0}}(p,r)\times[0,T_{\ast}]$.
    
    Under assumption $\eqref{5.5}$, the Shi-type estimate of Theorem 2.1 in \cite{Chen Shi-estimates} provides that
    \begin{align}
        t^{\frac{1}{2}}\left(|\nabla{\rm Rm}|+|\nabla^{2}\mathbf{T}|\right)\leq C(r,T_{\ast},\mathbf{M})
    \end{align}
     for all $(x,t)\in B_{g_{0}}(p,r)\times[0,T_{\ast}]$. Let $\Theta(x,t)=t\mathbf{M}^{2}|\nabla h(x,t)|^{2}$, then from the above yields
     \begin{align}
         \Theta(x,t)&\leq Ct\mathbf{M}^{2}\left(|\nabla{\rm Rm}|^{2}+|\mathbf{T}|^{2}|\nabla\mathbf{T}|^{2}+|\mathbf{T}|^{6}+A^{2}|\nabla\mathbf{T}|^{2}\right)\\
         &\leq Ct\mathbf{M}^{2}\left(|\nabla{\rm Rm}|^{2}+|\nabla\mathbf{T}|^{3}+|\mathbf{T}|^{6}+A^{6}\right)\notag\\
         &\leq C(r,T_{\ast},\mathbf{M})\notag
     \end{align}
     for all $(x,t)\in B_{g_{0}}(p,r)\times[0,T_{\ast}]$, where the second inequality uses Cauchy-Schwarz inequality. Thus we can construct the cut-off function satisfies $\eqref{5.6}$ via the proof of Part 2 of Lemma 14.3 in \cite{Ricci flow TA 2}.
\end{proof}

\begin{theorem}\label{theorem 5.4}
     Let $(U,\psi(t))_{t\in[0,T]}$ be a smooth solution of the modified Laplacian coflow $\eqref{modified Laplacian coflow}$ on an open set $U\subset M$. Let $p\in U$ and $r>0$ such that $\overline{B}_{g_{0}}(p,2r)\subset U$ is compact, where $g_{0}$ is the metric induced by initial coclosed $G_{2}$-structure $\psi$. Suppose that
     $$\mathbf{M}=\sup_{B_{g_{0}}(p,2r)\times[0,T]}\Lambda(x,t).$$
     Then there exist positive constants $L,C,T_{\ast}$ depending only on $r,\mathbf{M},T$ such that
     \begin{align}\label{5.10}
         t^{\frac{k}{2}}\left(|\nabla^{k}{\rm Rm}|(x,t)+|\nabla^{k+1}\mathbf{T}|(x,t)+|\nabla^{k+2}\psi|(x,t)\right)\leq CL^{\frac{k}{2}}(k+1)!
     \end{align}
     for all $k\in\mathbb{N}$ and $(x,t)\in B_{g_{0}}(p,r/2)\times[0,T_{\ast}]$.
\end{theorem}
\begin{proof}
    The proof is similar to the compact case, as defined in \cite{L-W 2019 2}, for $L>0$, we denote by
    \begin{align}
        \alpha_{k}&=\frac{\eta^{\frac{k+1}{2}}}{L^{\frac{k}{2}}}a_{k},\ \quad \beta_{k}=\frac{\eta^{\frac{k+1}{2}}}{L^{\frac{k}{2}}}b_{k},\quad\gamma_{k}=\frac{\eta^{\frac{k+1}{2}}}{L^{\frac{k}{2}}}c_{k},\quad \ \delta_{k}=\frac{\eta^{\frac{k+1}{2}}}{L^{\frac{k}{2}}}d_{k},\quad \text{for}\ k\geq 0,\notag\\
        \tilde{\alpha}_{k}&=\frac{\eta^{\frac{k}{2}}}{L^{\frac{k-1}{2}}}\tilde{a}_{k},\quad \tilde{\beta}_{k}=\frac{\eta^{\frac{k}{2}}}{L^{\frac{k-1}{2}}}\tilde{b}_{k},\quad \tilde{\gamma}_{k}=\frac{\eta^{\frac{k}{2}}}{L^{\frac{k-1}{2}}}\tilde{c}_{k},\quad \tilde{\delta}_{k}=\frac{\eta^{\frac{k}{2}}}{L^{\frac{k-1}{2}}}\tilde{d}_{k},\quad \text{for}\ k\geq 1,\notag
    \end{align}
    where $a_{k},b_{k},c_{k},d_{k},\tilde{a}_{k},\tilde{b}_{k},\tilde{c}_{k},\tilde{d}_{k}$ are defined in Section \ref{section4} and $\eta$ is the cut-off function from Lemma \ref{lemma 5.3}. We also define
    \begin{align}
        \tilde{\Phi}_{N}&=\sum_{k=0}^{N}\left(\alpha_{k}^{2}+\beta_{k}^{2}+\gamma_{k}^{2}+\delta_{k}^{2}\right)+|\mathbf{T}|^{2}+A^{2}+|\varphi|^{2}+|\psi|^{2},\\
        \tilde{\Psi}_{N}&=\sum_{k=0}^{N}\left(\tilde{\alpha}_{k}^{2}+\tilde{\beta}_{k}^{2}+\tilde{\gamma}_{k}^{2}+\tilde{\delta}_{k}^{2}\right).
    \end{align}
    The main proof is to estimate the evolution inequality of $\tilde{\Phi}_{N}$, thus we should estimate the evolution equations of $\alpha_{k}^{2},\beta_{k}^{2},\gamma_{k}^{2},\delta_{k}^{2},|\mathbf{T}|^{2}$. Firstly, we consider the case of $\alpha_{k}^{2}$
    \begin{align}\label{5.13}
        \left(\frac{\partial}{\partial t}-\triangle\right)\alpha_{k}^{2}&=\left(\frac{\partial}{\partial t}-\triangle\right)\left(\frac{\eta^{k+1}}{L^{k}}a_{k}^{2}\right)\\
        &= \frac{\eta^{k+1}}{L^{k}}\left(\frac{\partial}{\partial t}-\triangle\right)a_{k}^{2}+\frac{a_{k}^{2}}{L^{k}}\left(\frac{\partial}{\partial t}-\triangle\right)\eta^{k+1}\notag\\
        &\quad-\frac{2}{L^{k}}\langle\nabla\eta^{k+1},\nabla a_{k}^{2}\rangle.\notag
    \end{align}
    For the second term, from Lemma \ref{lemma 5.3}, one gets
    \begin{align}
        \frac{a_{k}^{2}}{L^{k}}\left(\frac{\partial}{\partial t}-\triangle\right)\eta^{k+1}\leq (k+1)C_{0}\frac{a_{k}^{2}}{L^{k}}\eta^{k}=\frac{C_{0}t}{L(k+1)}\tilde{\alpha}_{k}^{2}.
    \end{align}
    For the third term, by the Cauchy-Schwarz inequality, we have
    \begin{align}
        -\frac{2}{L^{k}}\langle\nabla\eta^{k+1},\nabla a_{k}^{2}\rangle&\leq 4\frac{\eta^{k}|\nabla\eta|}{L^{k}}\frac{t^{k}|\nabla^{k}{\rm Rm}||\nabla^{k+1}{\rm Rm}|}{k!(k+1)!}=4\frac{t^{\frac{1}{2}}\eta^{k}|\nabla\eta|}{L^{k}}\tilde{a}_{k}\tilde{a}_{k+1}\\
        &\leq \frac{1}{8}\tilde{\alpha}_{k+1}^{2}+\frac{128C_{0} t}{L}\tilde{\alpha}_{k}^{2}\notag.
    \end{align}
    To estimate the first term, with the same discussion in Lemma \ref{lemma 4.3} and note that $\eta\leq 1$, it follows that
    \begin{align}
        \sum_{k=0}^{N}\frac{\eta^{k+1}}{L^{k}}\left(\frac{\partial}{\partial t}-\triangle\right)a_{k}^{2}&\leq -2\sum_{k=0}^{N}\tilde{\alpha}_{k+1}^{2}+\frac{1}{4L}\sum_{k=0}^{N}\tilde{\alpha}_{k}^{2}+\frac{3}{16}\tilde{\Psi}_{N}\\
        &\quad+C\tilde{\Phi}_{N}^{4}\left(1+\frac{t}{L}\tilde{\Phi}_{N}\right)^{3}+C\tilde{\Psi}_{N}\left(\left(1+\frac{t}{L}\tilde{\Phi}_{N}\right)^{4}-1\right)\notag.
    \end{align}
    Submitting the above into $\eqref{5.13}$, we get
    \begin{align}\label{5.17}
        \left(\frac{\partial}{\partial t}-\triangle\right)\sum_{k=0}^{N}\alpha_{k}^{2}&\leq -\frac{15}{8}\sum_{k=0}^{N}\tilde{\alpha}_{k+1}^{2}+\left(\frac{C}{L}+\frac{3}{16}\right)\tilde{\Psi}_{N}\\
        &\quad+C\tilde{\Phi}_{N}^{4}\left(1+\frac{t}{L}\tilde{\Phi}_{N}\right)^{3}+C\tilde{\Psi}_{N}\left(\left(1+\frac{t}{L}\tilde{\Phi}_{N}\right)^{4}-1\right)\notag
    \end{align}
    on $U\times [0,T]$, where $C=C(r,T,\mathbf{M})$. By the same proof, we can also obtain the evolution inequality of $\displaystyle\sum_{k=0}^{N}\beta_{k}^{2}$,$\displaystyle\sum_{k=0}^{N}\gamma_{k}^{2}$ and $\displaystyle\sum_{k=0}^{N}\delta_{k}^{2}$, which can be derived as
    \begin{align}
        &\quad\left(\frac{\partial}{\partial t}-\triangle\right)\sum_{k=0}^{N}\left(\beta_{k}^{2}+\gamma_{k}^{2}+\delta_{k}^{2}\right)\\
        &\leq -\frac{15}{8}\sum_{k=0}^{N}\left(\tilde{\beta}_{k+1}^{2}+\tilde{\gamma}_{k+1}^{2}+\tilde{\delta}_{k+1}^{2}\right)+3\left(\frac{C}{L}+\frac{3}{16}\right)\tilde{\Psi}_{N}\notag\\
        &\quad+C\tilde{\Phi}_{N}^{5}\left(1+\frac{t}{L}\tilde{\Phi}_{N}\right)^{4}+C\tilde{\Psi}_{N}\left(\left(1+\frac{t}{L}\tilde{\Phi}_{N}\right)^{5}-1\right)\notag.
    \end{align}
    Combining Lemma \ref{lemma 3.5} and $\eqref{5.17}$. Choosing $L$ large enough such that $\frac{4C}{L}\leq\frac{1}{8}$, thus we get
    \begin{align}
        \left(\frac{\partial}{\partial t}-\triangle\right)\tilde{\Phi}_{N}\leq -\tilde{\Psi}_{N}+C\tilde{\Psi}_{N}\left[\left(1+\frac{t}{L}\tilde{\Phi}_{N}\right)^{5}-1\right]+C\tilde{\Phi}_{N}^{5}\left(1+\frac{t}{L}\tilde{\Phi}_{N}\right)^{4}
    \end{align}
    on $U\times [0,T]$. 
    
    Let $\widetilde{T}_{N}$ be the time
    \begin{align}
        \widetilde{T}_{N}:=\sup\left\{a\in[0,T]\left|\frac{t}{L}\tilde{\Phi}_{N}\leq (1+C^{-1})^{\frac{1}{5}}-1, \text{for any}\ (x,t)\in U\times[0,a]\right.\right\}.\notag
    \end{align}
   Following the proof of Theorem \ref{theorem 4.8}, there {exists} a time $t_{1}:=t_{1}(r,T,\mathbf{M})$ satisfies the definition of $\widetilde{T}_{N}$, and we denote
    \begin{align}
        T_{\ast}=T_{\ast}(r,T,M_{0}):=\min\left\{T,t_{1},1\right\}>0,
    \end{align}
    Thus we obtain $\eqref{5.10}$ on $B_{g_{0}}(p,r/2)\times[0,T_{\ast}]$ for all $k\in\mathbb{N}$.
\end{proof}

\begin{theorem}
    Let $(U,\psi(t),g_{\psi}(t))_{t\in[0,T]}$ be a smooth solution of the modified Laplacian coflow $\eqref{modified Laplacian coflow}$ with coclosed $G_{2}$-structure on an open set $U\subset M$, then for each time $t\in(0,T]$, $(U,\psi(t),g_{\psi}(t))$ is real analytic.
\end{theorem}
\begin{proof}
    With the same discussion in Theorem \ref{theorem 5.2}, the local real analyticity follows from Theorem \ref{theorem 5.4}.
\end{proof}

$${}$$

\textbf{Acknowledgments.}\ \ 
The first author would like to thank his advisor Professor Jian Ge for his guidance. He also would like to thank Professor Kotaro Kawai for reviewing the original manuscript, and thank Nanhao Gao for pointing out some errors. 
The second author is funded by Shanghai Institute for Mathematics and Interdisciplinary Sciences (SIMIS) under grant number SIMIS-ID-2025-AD.  The authors would also like to thank the referee for the valuable comments and suggestions.

\bibliographystyle{amsplain}

\end{document}